\documentclass{amsart}

\usepackage{fullpage}
\usepackage{amsmath}
\usepackage{amsfonts}
\usepackage{mathtools}
\usepackage{amscd,amssymb,color}
\usepackage{hyperref}

\usepackage{algorithm}
\usepackage{algpseudocode}

\usepackage{tikz}
\usepackage{verbatim}
\usetikzlibrary{matrix}

\usepackage{amsthm}

\numberwithin{equation}{section}

\newtheorem{lemma}[equation]{Lemma}
\newtheorem{definition}[equation]{Definition}
\newtheorem{theorem}[equation]{Theorem}
\newtheorem{proposition}[equation]{Proposition}
\newtheorem{corollary}[equation]{Corollary}
\theoremstyle{remark}
\newtheorem{remark}[equation]{Remark}

\newcommand{\dpp}{\tilde d_{\mathcal{A},p}}
\newcommand{\dqq}{\tilde d_{\mathcal{A},q}}
\newcommand{\dinf}{\tilde d_{\mathcal{A},\infty}}
\renewcommand{\d}{d_{\mathcal{GH}}}

\begin{document}

\title{A polynomial-time relaxation of the Gromov-Hausdorff distance}

\author[Villar]{Soledad Villar}
\address{Department of Mathematics, University of Texas,
Austin, TX \ 78703}
\email{mvillar@math.utexas.edu}

\author[Bandeira]{Afonso S. Bandeira}
\address{Department of Mathematics and Center for Data Science, Courant Institute of Mathematical Sciences, New York University, 
New York, NY \ 10012}
\email{bandeira@cims.nyu.edu}

\author[Blumberg]{Andrew J. Blumberg}
\address{Department of Mathematics, University of Texas,
Austin, TX \ 78703}
\email{blumberg@math.utexas.edu}

\author[Ward]{Rachel Ward}
\address{Department of Mathematics, University of Texas,
Austin, TX \ 78703}
\email{rward@math.utexas.edu}

\thanks{A.~J.~Blumberg and S.~Villar were supported in part by AFOSR grant FA9550-15-1-0302.}

\thanks{R. Ward and S. Villar were supported in part by NSF CAREER grant $\#1255631$}

\thanks{Part of this work was done while A.~S.~Bandeira was with the Mathematics Department at MIT and supported by NSF Grant DMS-1317308}

\begin{abstract}
The Gromov-Hausdorff distance provides a metric on the set of isometry
classes of compact metric spaces.  Unfortunately, computing this
metric directly is believed to be computationally intractable.  Motivated by
applications in shape matching and point-cloud comparison, we study a
semidefinite programming relaxation of the Gromov-Hausdorff metric.
This relaxation can be computed in polynomial time, and somewhat
surprisingly is itself a pseudometric.  We describe the induced
topology on the set of compact metric spaces.  Finally, we demonstrate
the numerical performance of various algorithms for computing the
relaxed distance and apply these algorithms to several relevant
data sets. In particular we propose a greedy algorithm for finding the best correspondence between finite metric spaces that can handle hundreds of points.
\end{abstract}

\maketitle

\section{Introduction}

In order to study the convergence of sequences of metric spaces,
Gromov introduced what is now called the Gromov-Hausdorff
metric~\cite{gromov81}.  Roughly speaking, this metric generalizes the
classical Hausdorff distance between subsets of an ambient metric
space, to a pair of arbitrary metric spaces by taking the infimum over
all embeddings.  The Gromov-Hausdorff metric has been of theoretical
importance in geometric group theory and is at the heart of the
subject of ``metric geometry''.

More recently, the Gromov-Hausdorff distance has been proposed as a basic
method for comparing {\em point clouds}~\cite{memoli05}.  A point
cloud is simply a finite metric space (often presented as a subset
of $\mathbb{R}^n$); this is a fundamental and ubiquitous
representation of data.  Geometric examples, where the point cloud
represents samples from some smooth geometric object, arise from
various kinds of shape acquisition devices.  Examples with less
obvious intrinsic geometric structure are frequently generated by
biological data (e.g., collections of gene expression vectors).  Given
two point clouds, a natural question is to determine if they are
related by some isometric transformation; if not, one might wish to
know a quantitative measure of their difference.

Another version of this sort of problem is what is known as the
point registration problem (also sometimes referred to as point
matching and network alignment).  Point registration consists in
finding a correspondence between point sets or graphs such that a
certain cost function is minimized. It appears in computer vision
problems like shape matching, computational biology
\cite{Clark02052014}, and general pattern recognition problems.
In some applications, registering or aligning is particularly
challenging since there is no explicit correspondence between the
sets, often because deformation has occurred or they have different
numbers of points.  In such cases it is natural to consider a metric on
point clouds that is defined in terms of correspondences between point
clouds together; the Gromov-Hausdorff distance can be described in terms
of a minimax expression over correspondences between the metric
spaces, and so is potentially suitable for this purpose.

Unfortunately, exact computation of the Gromov-Hausdorff distance is
essentially intractable; it involves the solution of an NP-hard
optimization problem.  As a consequence, it is natural to consider
relaxations.  In~\cite{memoli11}, M\'emoli studied a relaxation
referred to as the Gromov-Wasserstein distance --- this distance is
closely related to distances motivated by optimal transport
problems~\cite{LottVillani, Sturm}, to a ``distance distribution''
metric defined by Gromov, and also to the cut distance of graphons.
Unfortunately, computing the Gromov-Wasserstein distance still
requires solving a non-convex optimization problem which does not
appear to have attractive performance characteristics in practice.

In this paper, we study a semidefinite programming (SDP) relaxation of
the Gromov-Hausdorff distance; this yields a tractable convex optimization
problem.  We prove that this relaxation defines a pseudometric on
point clouds that can be computed in polynomial time.  Our
pseudometric also provides a relaxed correspondence between point
clouds and a lower bound for the Gromov-Hausdorff distance.  A
similar version of this SDP was recently introduced in \cite{kezurer15} and further studied in \cite{Lipman} and \cite{Lipman16};
our work provides theoretical validation for some of the computational 
phenomena observed therein and complements their theoretical framework.  We describe the performance of optimized
solvers for this SDP.  We also propose a non-convex optimization
algorithm to approach the registration problem efficiently. The output
of this algorithm not only provides a local optimum for the
registration problem, but also an upper bound for the Gromov-Hausdorff
distance.

\section{Background}

We begin by recalling the definition of the Gromov-Hausdorff distance;
for this, we start with the Hausdorff distance.  Let $(Z,d)$ a compact
metric space and $\mathcal C(Z)$ the collection of all compact sets in
$Z$.  If $A,B \in \mathcal C(Z)$, the Hausdorff distance between $A$
and $B$ can be expressed as
\[
d_{\mathcal H}^Z(A,B)= \inf_{R\in\mathcal R(A,B)}\sup_{(a,b)\in R}
d(a,b)
\]
where $\mathcal R (A,B)$ is the set of correspondences in
$R\subset A \times B$ such that every element $a\in A$ is related to
at least one element in $B$ and every element $b\in B$ is related at
least one element in $A$. For many theoretical and practical applications
it is common to relax such distance to the Wasserstein distance
\cite{villani}. In that setting, one endows $\mathcal C(Z)$ with a
measure, $$\mathcal C_{w}(Z)=\{(A,\mu_A)\colon A\in \mathcal C(Z), \;
\operatorname{supp}(\mu_A)=A\},$$ and relaxes the set of
correspondences $\mathcal R(A,B)$ to the set of transportation plans
\[
\mathcal M(\mu_A,\mu_B)=\{\mu\colon \mu(A_0\times B)=\mu_A(A_0), \;
\mu(A\times B_0)=\mu_B(B_0), \text{ for all Borel sets } A_0\subset A,
\,B_0\subset B \}.
\]
Then, for $A,B\in \mathcal C_w(Z)$, the
Wasserstein distance is defined for $1\leq p \leq \infty$ as
\[
d_{\mathcal W, p}^Z(A,B)=\inf_{\mu\in \mathcal
  M(\mu_A,\mu_B)}\left(\int_{A\times B}d^p(a,b)d\mu(a,b)\right)^{1/p}
\quad \text {for } 1\leq p < \infty
\]
\[
d_{\mathcal W, \infty}^Z(A,B)=\inf_{\mu\in \mathcal
  M(\mu_A,\mu_B)}\sup_{(a,b)\in \operatorname{supp}(\mu)}d(a,b).
\] 
For $A,B$ finite sets this distance can be efficiently computed by a
linear program.

The Hausdorff distance suffices to compare metric spaces embedded in
a common ambient metric space; Gromov's idea to extend this to compare
arbitrary metric spaces is simply to consider the infimum over all
isometric embeddings into a common metric space~\cite{gromov}.
Specifically, if $X,Y$ are compact metric spaces, the Gromov-Hausdorff
distance is defined as 
\[
d_{\mathcal{GH}}(X,Y)=\inf_{Z,f,g}d_{\mathcal H}^Z(f(X),g(Y))
\]
where $f\colon X\to Z$ and $g\colon Y \to Z$ are isometric embeddings
into $Z$, a metric space.  Unfortunately, it is an NP-hard problem to
compute this distance.

Since the Hausdorff distance becomes computationally tractable when
relaxed to the Wasserstein distance, one might consider a
transport-based relaxation of the Gromov-Hausdorff distance that works
in the setting of metric measure spaces.  In a series of papers
\cite{memoli11, memoli07, memoli05} M\'emoli considers different
equivalent expressions for the Gromov-Hausdorff distance, and by
relaxing them and considering them in the measure metric space
setting, he obtains different \emph{gromovizations} of the Wasserstein
distance, called Gromov-Wasserstein distances.  A particularly natural
relaxation is based on the observation that the Gromov-Hausdorff
distance can be expressed as: 
\begin{equation} \label{gh_cont}
d_{\mathcal{GH}}(X,Y)=\frac{1}{2} \inf_{R\in\mathcal R(X,Y)} \sup_{\begin{subarray}{c} x,x'\in X \\ y,y'\in Y \\ (x,y),(x',y')\in R  \end{subarray}} \Gamma_{X,Y}(x,y,x',y')
\end{equation}
where $\Gamma_{X,Y}(x,y,x',y')=|d_X(x,x')-d_Y(y,y')|$. 
For $1\leq p\leq \infty$ M\'emoli then defines Gromov-Wasserstein
relaxations of the Gromov-Hausdorff distance as   
\begin{equation} \label{gw_cont}
D_{p}(X,Y)=\frac{1}{2} \inf_{\mu\in \mathcal M(\mu_X,\mu_Y)}\left( \int_{X\times Y}\int_{X\times Y}\left(\Gamma_{X,Y}(x,y,x',y')\right)^p \mu(dx\times dy)\mu(dx'\times dy') \right)^{1/p}
\end{equation}
\begin{equation} \label{gw_infinity}
D_{\infty}(X,Y)=\frac{1}{2} \inf_{\mu\in \mathcal M(\mu_X,\mu_Y)} \sup_{\begin{subarray}{c} x,x'\in X \\ y,y'\in Y \\ (x,y),(x',y')\in \operatorname{supp}(\mu)  \end{subarray}}\Gamma_{X,Y}(x,y,x',y')
\end{equation}

In his work, M\'emoli studies topological properties of the different
distance relaxations and how they compare with each other and with the
Gromov-Hausdorff distance. He also proposes an algorithm to
approximate $D_p$, but due the non-convexity of its objective
function, no performance guarantees are provided. Recent work \cite{2016-peyre-icml} has provided efficient heuristic algorithms based on optimal transportation to approximate the Gromov-Wasserstein distance for alignment applications. 

\begin{remark}
Gromov considered another metric on the set of metric measure spaces
defined in terms of the convergence of all distance matrix
distributions (i.e., the distributions induced by taking the
pushforward of the measure to a collection of $n$ points and applying
the metric to all pairs).  It turns out that this metric is closely
related to the Gromov-Wasserstein distance~\cite[3.7]{Sturm}.
Moreover, these metrics induce the same notion of convergence as
arises in the theory of dense graph sequences and graphons.
Specifically, we can regard a graph as a metric measure space; the
underlying metric space has points the set of vertices and pairwise
distances $\frac{1}{2}$ if the points are connected and $1$ otherwise,
and the measure assigns equal mass to each point.  (See~\cite{Elek}
for further discussion of this point).
\end{remark}

\section{Semidefinite relaxations of Gromov-Wasserstein and Gromov-Hausdorff distances}
\renewcommand{\theenumi}{\alph{enumi}}

Consider the setting where $X$ and $Y$ are finite metric spaces (or metric measure spaces), say $X=\{x_1,\ldots, x_n\}$ and $Y=\{y_1,\ldots, y_m\}$ (with measures $\mu_X(x_i)=\nu_i$ and $\mu_Y(y_j)=\lambda_j$). 
Let us abbreviate $\Gamma_{X,Y}(x_i,y_j,x_{i'},y_{j'})$ as $\Gamma_{ij,i'j'}$ for $i,i'=1,\ldots, n$ and $j,j'=1,\ldots,m$. The formulation of the Gromov-Hausdorff distance given in equation~\eqref{gh_cont} can be expressed as a quadratic assignment (Remark 3 in \cite{memoli07}):
\begin{align}\label{gh}
d_{\mathcal{GH}}(X,Y)=& \frac{1}{2} \min_{\mu}\max_{ij,i'j' } \Gamma_{ij,i'j'} \mu_{ij}\mu_{i'j'}\quad  \text{ subject to } \mu_{ij} \in \{0,1\} , \;\; \sum_{j=1}^m \mu_{ij}\geq 1, \;\; \sum_{i=1}^n \mu_{ij}\geq 1 
\end{align}
and the expressions \eqref{gw_cont} and \eqref{gw_infinity} can be written as \eqref{gw} and \eqref{gwi} respectively:
\begin{equation} \label{gw}
D_p(X,Y)=\frac12\min_{\mu\in \mathbb R^{n\times m}} \left( \sum_{i,i'=1}^n\sum_{j,j'=1}^m \mu_{ij}\mu_{i'j'} \Gamma_{ij,i'j'}^p\right)^{1/p}  \text{subject to } 0\leq \mu_{ij} \leq 1,  \sum_{i=1}^n\mu_{ij}=\lambda_j,  \sum_{j=1}^m \mu_{ij}=\nu_i
\end{equation}
\begin{equation} \label{gwi}
D_\infty(X,Y)=\frac12\min_{\mu\in \mathbb R^{n\times m}} \max_{\stackrel{i,i',j,j'}{\mu_{ij}\mu_{i'j'}\neq 0} }  \Gamma_{ij,i'j'} \,\text{ subject to } 0\leq \mu_{ij} \leq 1,  \sum_{i=1}^n\mu_{ij}=\lambda_j,  \sum_{j=1}^m \mu_{ij}=\nu_i
\end{equation}

In order to approach non-convex optimization problems like \eqref{gh}, \eqref{gw} or \eqref{gwi}, one standard technique is to linearize the objective by lifting  $\mu_{ij}\mu_{i'j'}$ and $\mu_{ij}$  to a symmetric variable $\mathbf Z \in R^{nm+1 \times nm+1}$ whose entries entries are indexed by pairs $(ij,i'j')$,$(ij,nm+1)$, $(nm+1, i'j')$ and $(nm+1,nm+1)$ with $i,i'=1,\ldots n$ and $j,j'=1,\ldots,m$.
\begin{equation}
\label{Z}
\mathbf Z=\left[ \begin{array}{c c} \hat{\mathbf Z }& \mathbf z \\ \mathbf z^\top &1 \end{array} \right].
\end{equation}

Then note that, for instance, the problems \eqref{gw} and \eqref{gwi} are equivalent to problems \eqref{gw_lift} and \eqref{gwi_lift} respectively:
\begin{align} \label{gw_lift}
D_p(X,Y)=\frac{1}{2}\left(\min_{\mathbf Z} \operatorname{Trace}(\Gamma^{(p)} \hat{\mathbf Z})\right)^{1/p}\quad \text{ subject to } \mathbf Z \in \mathcal S 
\\
D_\infty(X,Y)=\frac{1}{2}\min_{\mathbf Z} \max_{i,i',j,j'\colon \mathbf{Z}\neq 0} \Gamma_{ij,i'j'}  \quad \text{ subject to } \mathbf Z \in \mathcal S \label{gwi_lift}
\end{align}
where $\mathcal S = \{ \mathbf Z \in \mathbb R^{nm+1\times nm+1}\colon \sum_{i}\mathbf Z_{ij,nm+1}=\lambda_j,\, \sum_{j}\mathbf Z_{ij,nm+1}=\nu_i, \,
{\mathbf Z}={\mathbf Z}^\top,\,
{\mathbf Z}\geq 0,\, \operatorname{rank}({\mathbf Z})=1 \}$ and $\Gamma^{(p)}$ denotes the $p$-th power of the matrix $\Gamma$ entrywise.

The constraint $\operatorname{rank}(\mathbf Z) =1$ can be relaxed to the convex constraint ${\mathbf Z} \succeq 0$ (which means ${\mathbf Z}$ is symmetric and positive semidefinite) and additional linear constraints satisfied by the rank 1 matrix can be added to make the relaxation tighter.

Using this recipe we can construct the following family of semidefinite programming relaxations of the Gromov-Wasserstein and Gromov-Hausdorff distances.

\begin{align} \label{gh_p}
\tilde d_{\mathcal{A},p}(X,Y)=\frac12 \left( \frac1{n^2}\min_{\mathbf Z}\operatorname{Trace}(\Gamma^{(p)}{\hat{\mathbf Z}} ) \right)^{1/p}\quad \text{ subject to } 
\mathbf Z \in \mathcal{ A}
\\
\tilde d_{\mathcal A, \infty}(X,Y)=\frac12\min_{\mathbf Z } \max_{i,j,i',j'\colon \mathbf Z \neq 0} \Gamma_{ij,i'j'}\quad \text{ subject to } \label{gh_inf}
\mathbf Z \in {\mathcal A}
\end{align}
where we can consider different convex sets $\mathcal A$ as relaxing to different distances.

\begin{enumerate}

\item For a relaxation of the Gromov-Hausdorff distance (or Gromov-Wasserstein for uniform weights $\lambda_j=\nu_i=1/\max\{n,m\}$ for all $j=1,\ldots,m$ and $i=1,\ldots,n$)\footnote{By appropriately choosing the right hand side of the equality constraints in $\mathcal{GH}$ one can obtain a semidefinite relaxation of Gromovo-Wasserstein distance for any weights.} consider 
\begin{multline*}
{\mathcal A}=\mathcal{GH}\colon=\{ \mathbf Z \in \mathbb R^{nm+1\times nm+1}\;\colon\quad \sum_{i}{\mathbf Z}_{ij,nm+1}\geq 1,\quad
\sum_{j}{\mathbf Z}_{ij,nm+1}\geq 1, \quad 
\sum_{i,i'} {\mathbf Z}_{ij,i'j'}\geq 1, \; \\
\sum_{j,j'} {\mathbf Z}_{ij,i'j'}\geq 1, \quad
\mathbf{ \hat Z 1}=\max\{n,m\} \mathbf z, \quad
0\leq {\mathbf Z}\leq 1,\quad 
{\mathbf Z}\succeq 0\}.
\end{multline*}
Relaxation \eqref{gh_inf} provides a lower bound for the Gromov-Hausdorff distance, since every element of $\mathcal R(X,Y)$ induces, up to normalization, a feasible $\mathbf Z$. In fact, if the optimal solution of equation~\eqref{gh_cont} corresponds to $R\in\mathcal R(X,Y)$ such that $(x_i,y_j),(x_i,y_{j'})\in R$ for some $j\neq j'$, the solution of equation~\eqref{gh_p} may split the mass in a way so $\mathbf Z_{ij,nm+1}+\mathbf Z_{ij',nm+1}= 1$ instead of having $\mathbf Z_{ij,nm+1}=\mathbf Z_{ij',nm+1}=1$.

\item If $|X|=|Y|$ we may want to restrict the set of all correspondences between $X$ and $Y$ (where every element of $X$ is related to at least one element in $Y$ and vice versa) to the set of all bijective correspondences. In that case we can consider a tighter relaxation, that relaxes the registration problem and is similar to the one in \cite{kezurer15}.
\begin{multline*}
\mathcal A=\operatorname{Reg}\colon= \{
 \mathbf Z \in \mathbb R^{n^2+1\times n^2+1}\colon
\displaystyle\sum_{i=1}^n{\mathbf Z}_{ij,n^2+1}= 1, \quad 
\displaystyle\sum_{j=1}^n{\mathbf Z}_{ij,n^2+1}= 1, \quad
\mathbf{Z}_{n^2+1,n^2+1}=1,
\\
\displaystyle\sum_{i,i'=1}^n \mathbf{Z}_{ij,i'j'}= 1, \quad 
\displaystyle\sum_{j,j'=1}^n \mathbf{Z}_{ij,i'j'}= 1, \quad
\mathbf Z_{ij,ij'}=0 \text { if } j\neq j', \quad
\mathbf Z_{ij,ij'}=0 \text { if } i\neq i', 
\\
\operatorname{Trace}({\mathbf Z})=n+1, \quad
\mathbf{ \hat Z 1}=n \mathbf z, \quad
0\leq {\mathbf Z}\leq 1, \quad
\mathbf Z\succeq 0 
\}
\end{multline*}

\item The registration relaxation can be extended to finite metric spaces with different numbers of points. Let as before $|X|=n$ and $|Y|=m$. First, without loss of generality assume that $n\geq m$. Now consider the problem \eqref{gh_cont} where the set $\mathcal R(X,Y)$ is restricted to surjective functions $X\to Y$. Then relax the feasible set to the convex set:
 \begin{multline*}\mathcal A= \operatorname{Sur}\colon=\{ \mathbf Z \in \mathbb R^{nm+1\times nm+1}\colon
\displaystyle\sum_{i=1}^n{\mathbf Z}_{ij,nm+1}\geq 1, \quad
\displaystyle\sum_{j=1}^m{\mathbf Z}_{ij,nm+1}= 1, \quad
\mathbf{Z}_{nm+1,nm+1}=1,
\\
\displaystyle\sum_{i,i'=1}^n \mathbf{Z}_{ij,i'j'}\geq 1, \quad
\displaystyle\sum_{j,j'=1}^m \mathbf{Z}_{ij,i'j'}= 1, \quad
\mathbf Z_{ij,ij'}=0 \text { if } j\neq j', 
\\
\operatorname{Trace}({\mathbf Z})=n+1,\quad
\mathbf{ \hat Z 1}=n \mathbf z, \quad
0\leq {\mathbf Z}\leq 1, \;\;
\mathbf Z\succeq 0, 
\}
\end{multline*}
Note that the set of constraints assumes that $i,i'=1,\ldots n$, $j,j'=1,\ldots m$ and $n\geq m$ and it is not symmetric with respect to $i$ and $j$. Also note that under this relaxation, there exist sets that $\d(X,Y)=0$ but the relaxed distance \eqref{gh_p} with $\mathcal A=\operatorname{Sur}$ satisfies $\tilde d_{\operatorname{Sur},p}(X,Y)\neq 0$ (and the same phenomena occurs for $\mathcal A = \operatorname{Reg}$). For instance $X=\{x,x,y\}$ and $Y=\{x,y,y\}$. This is an artifact of only allowing surjective functions instead of all possible relations in $\mathcal R(X,Y)$.

\end{enumerate}

\begin{remark}Even though the $\max$ objective in equation~\eqref{gh_inf} is convex, it is not smooth, which we observe to significantly hurt the performance of the numerical implementations. This is one reason to consider the $p$-norm approach to this relaxation and define the family of SDP relaxations \eqref{gh_p} for $1\leq p < \infty$.
\end{remark}

\begin{remark}
Note that linear constraints in the sets $\mathcal A$ are not linearly independent and the extra variable $\mathbf z$ is redundant. However, one can easily choose alternative sets of constraints (even with fewer constraints or where the extra constraint is not redundant) with the same objectives in mind:
(i) the relaxation is tight when the spaces are isometric (ii) the corresponding objective value satisfies the triangle inequality when $|X|=|Y|$. 
\end{remark}

We are primarily interested in the semidefinite programming relaxations of the Gromov-Hausdorff distance for finite metric spaces, namely $\dpp$, $1\leq p\leq \infty$ for $\mathcal A= \mathcal{GH}, \operatorname{Reg}, \operatorname{Sur}$.  In figure~\ref{diagram}, we extend a diagram of M\'emoli's to situate the SDP relaxations we study in this paper.


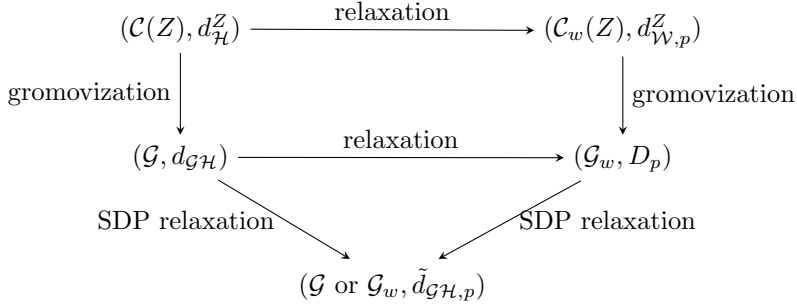
\begin{figure}[h]
\begin{center}
\begin{tikzpicture}
  \matrix (m) [matrix of math nodes,row sep=3em,column sep=1.5em,minimum width=2em]
  {
     (\mathcal C(Z), d_{\mathcal H}^Z)& & (\mathcal C_w(Z), d_{\mathcal W, p}^Z) \\
     (\mathcal G, d_{\mathcal{GH}})& & (\mathcal G_w, D_p) \\
     &(\mathcal G \text{ or }\mathcal G_w, \tilde d_{\mathcal{GH},p}) & 
     \\ };
  \path[-stealth]
    (m-1-1) edge node [left] {gromovization} (m-2-1)
            edge node [above] {relaxation} (m-1-3)
    (m-2-1) edge node [left] {SDP  relaxation} (m-3-2)
            edge node [above] {relaxation} (m-2-3)
    (m-1-3) edge node [right] {gromovization} (m-2-3)
    (m-2-3) edge node [right] { SDP  relaxation} (m-3-2)
            ;
\end{tikzpicture}
\caption{\label{diagram}  Here $\mathcal G$ is the collection of all compact metric spaces and $\mathcal G_w$ the collection of all metric measure spaces. The horizontal arrows represent the relaxation on the notion of correspondences. The gromovization arrows represent the process of getting rid of the ambient space.}
\end{center}
\end{figure}

\section{Topological properties of the relaxed distances} 

In this section, we prove the main theoretical results of the paper.
We begin by showing that the distances obtained by semidefinite
relaxation are in fact pseudometrics on suitable subsets of the set of
isometry classes of finite metric spaces; i.e., these distances
satisfy all the axioms for a metric except that there exist distinct
finite metric spaces such that the relaxed distance between them is
$0$.  We then study various properties of the induced topology,
proving analogues of the standard results about the topology induced
on the set of isometry classes of compact metric spaces by the
Gromov-Hausdorff distance.

\subsection{Pseudometrics} \label{sec:semimetric}

Let $\mathcal G_{<\infty}$ the set of all finite metric spaces.
First, we observe that $\dinf$ is a pseudometric in $\mathcal
G_{<\infty}$.  However, if the spaces $X,Y,Z$ have different numbers
of points we cannot expect the triangle inequality to hold for
$\dpp$. That is because the triangle inequality does not even hold for
tight solutions of equation~\eqref{gh_p} (i.e., rank 1 solutions,
corresponding to elements of $\mathcal R(X,Y)$). This is an
artifact of replacing the $\max$ with a sum.
 
In order to illustrate that fact we consider a simple example. Let $d_{\mathcal{GH},1}$ be the optimal of \eqref{gh_p} for $p=1$ and $\mathcal A$ the domain of \eqref{gh} (i.e. the solutions corresponding to elements of $\mathcal R(X,Y)$ ). Then consider $X=\{x,y\}$, $Y=\{x, x, y\}$, $Z=\{y\}$, and observe that triangle inequality is not satisfied since $d_{\mathcal{GH},1}(X,Y)=0$, $d_{\mathcal{GH},1}(Y,Z)=2d(x,y)$ and $d_{\mathcal{GH},1}(X,Z)=d(x,y)$.

Nonetheless, if we consider the set of metric spaces with $n$ points,
which we denote by $\mathcal G_n$, we will show that $\dpp$ for $1\leq
p <\infty$ is a pseudometric on $\mathcal G_n$.  The most interesting
part of this verification is the triangle inequality, which we prove
in Theorem~\ref{triangle} below.  In contrast to the situation with
the Gromov-Hausdorff distance, passing to isometry classes of finite
metric spaces does not suffice to produce an actual metric.  Of
course, if $X$ and $Y$ are isometric spaces then $\tilde
d_{\mathcal{A},p}(X,Y)=0$.  By construction $\tilde
d_{\mathcal{A},p}(\cdot,\cdot)\geq 0$ and the isometry between $X$ and
$Y$ induces a feasible solution for equation~\eqref{gh_p} with objective value
0. However, there exists non-isometric spaces $X,Y$ such that $\tilde
d_{\mathcal{A},p}(X,Y)=0$.  Examples of that phenomenon can be
constructed by observing that the graph isomorphism problem can be reduced to
deciding whether the Gromov-Hausdorff distance is zero.  Given a graph
$G=(V, E)$ one then constructs a metric space $X(G)$ where  
\begin{equation} \label{spaces}
d(v,v')= 
\left\{ 
\begin{matrix}1 & \text{if } (v,v') \in E \\
K\gg |V| & \text{otherwise}.
\end{matrix}
\right.
\end{equation}
Therefore, given two graphs $G,G'$ we have that $G,G'$ are isomorphic
if and only if $d_{\mathcal {GH}}(X(G), X(G') )=0$.  There exist
explicit examples in the literature of graphs where any SDP relaxation
on $|V|^2\times |V|^2$ matrices cannot distinguish between two 
non-isomorphic graphs~\cite{odonell}.  For such examples, $\tilde
d_{\mathcal {A},p}(X(G), X(G') )=0$ (see Figure~\ref{lasserre}).

\newsavebox{\smlmat}
\savebox{\smlmat}{\smallskip$ \quad\quad \left.\begin{matrix} 
x_1  \oplus   x_2  \oplus x_5 = b_1 \\
 x_1  \oplus x_2 \oplus x_5 = b_2 \\
 x_1   \oplus x_3\oplus    x_4= b_3\\
 x_2  \oplus x_3  \oplus x_4 = b_4 
 \end{matrix}\right., \text{ with } 
 \left(\begin{matrix} 
 b_1 \\
 b_2 \\
 b_3\\
 b_4
 \end{matrix} \right) =
  \left(\begin{matrix} 
 1 \\
 1 \\
 0\\
 1
 \end{matrix} \right) \text{(left), and }
 \left(\begin{matrix} 
 b_1 \\
 b_2 \\
 b_3\\
 b_4
 \end{matrix} \right) =
  \left(\begin{matrix} 
 0 \\
 0 \\
 0\\
 0
  \end{matrix} \right) \text{ (middle).}
 $}

\begin{figure}[b]
\includegraphics[width=0.32\textwidth, trim={6cm 9cm 6cm 9cm},clip]{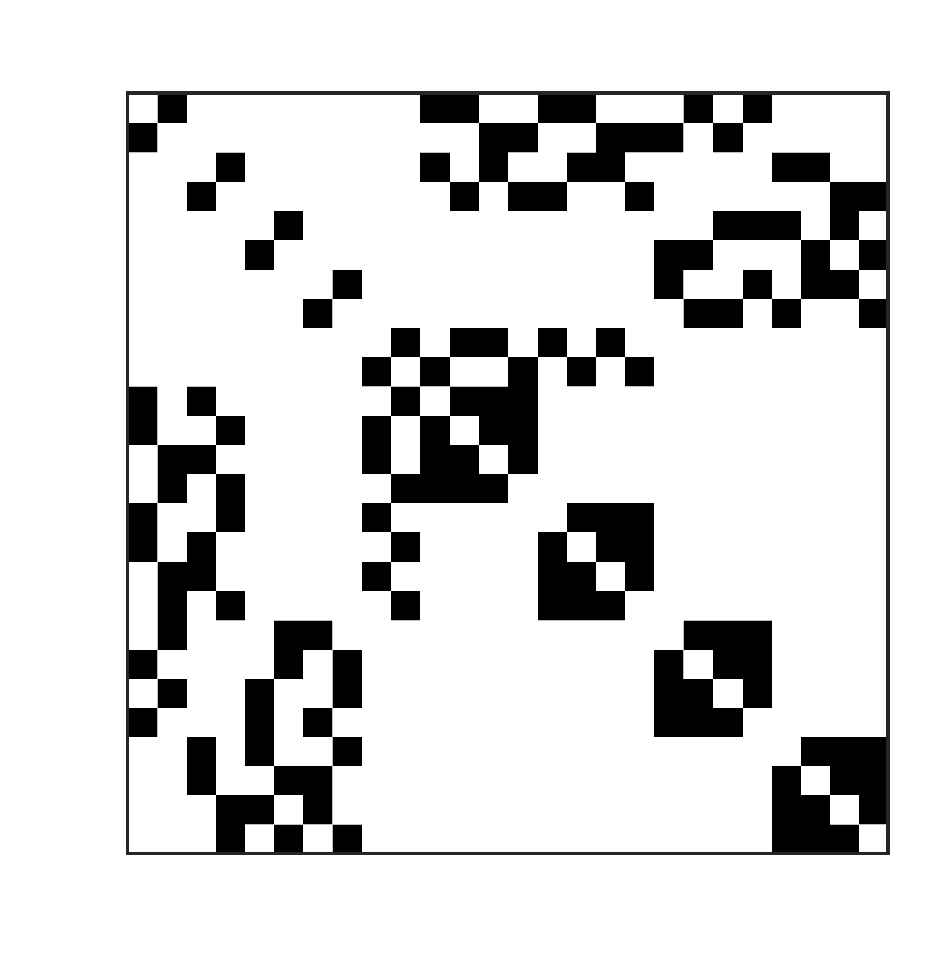}
\includegraphics[width=0.32\textwidth, trim={6cm 9cm 6cm 9cm},clip]{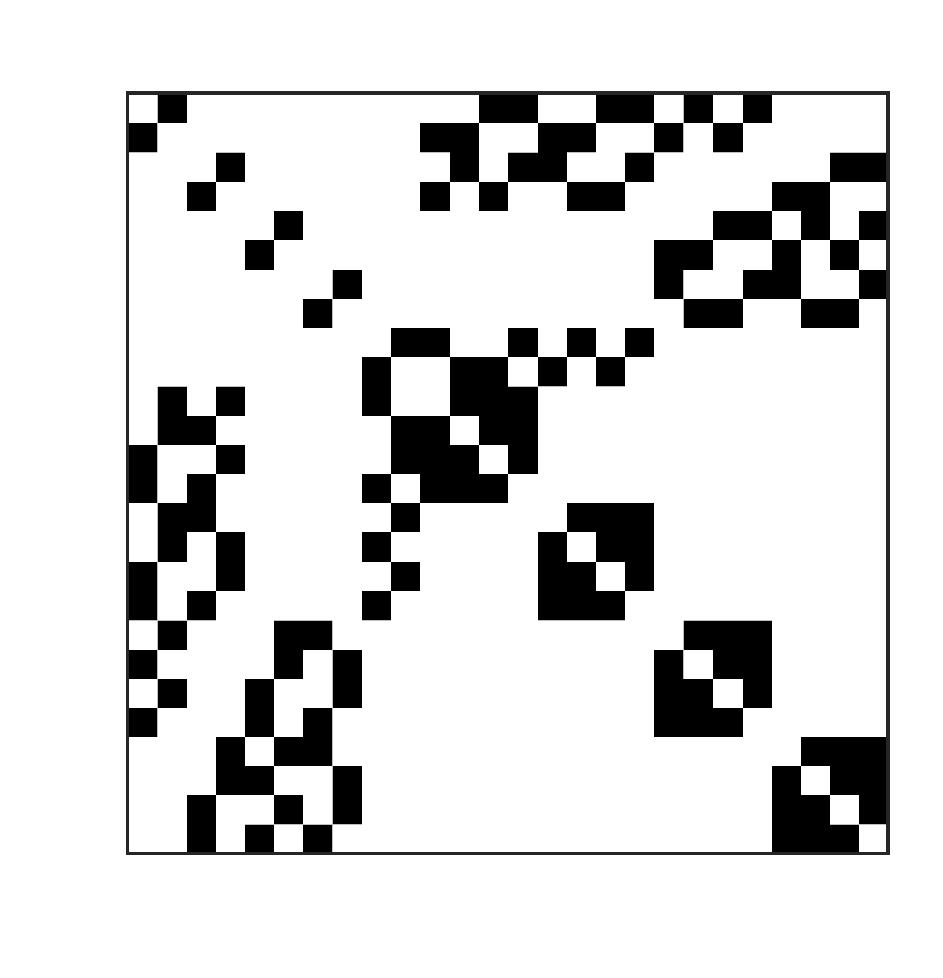}
\includegraphics[width=0.32\textwidth, trim={6cm 9cm 6cm 9cm},clip]{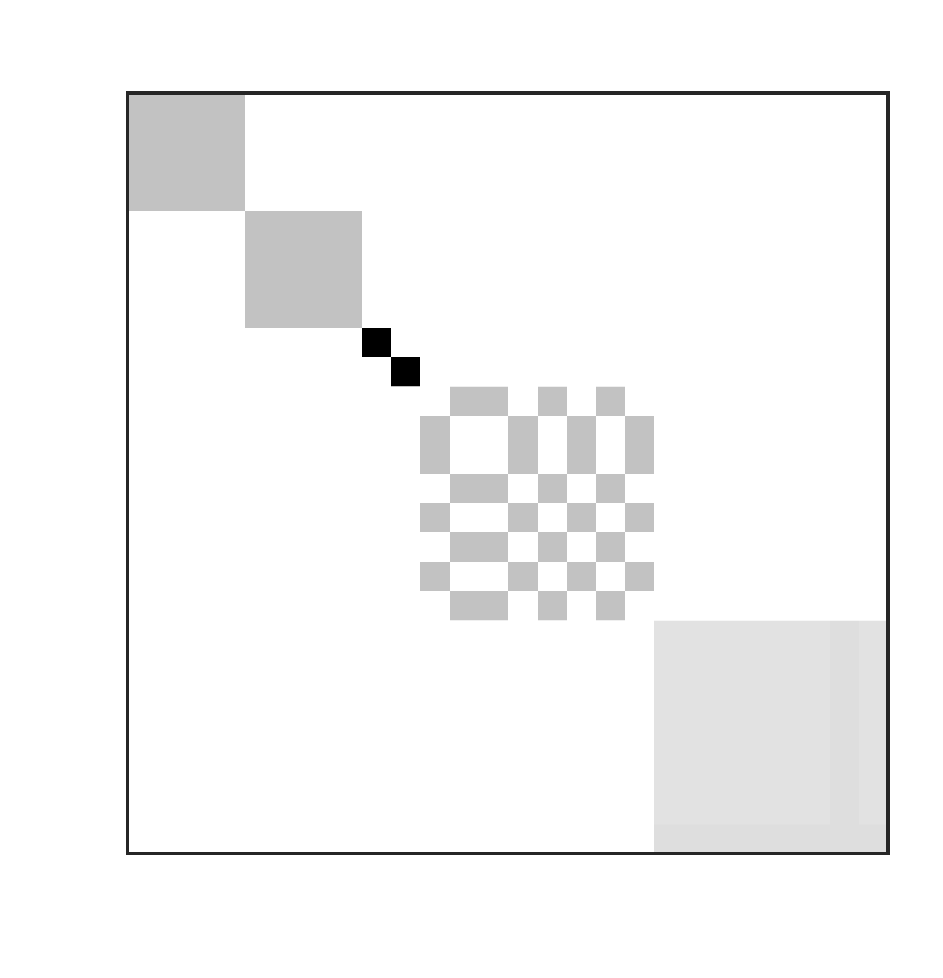}
\caption{\label{lasserre} We consider 3XOR instances with 5 variables and 4 equations and we construct the reduction from 3XOR to graph isomorphism from \cite{odonell}. The left and middle figures represent corresponding adjacency graphs obtained after the reduction from the following system of equations in $\mathbb Z_2$: 
\newline
\usebox{\smlmat} 
\newline
 Each graph has 26 vertices. We construct finite metric spaces $X$ and $Y$ according to \eqref{spaces} and we use SDPNAL+ \cite{yang2015sdpnal+} to compute the  the relaxed distance, obtaining $\tilde d_{\operatorname{Reg},1}(X,Y)=0$. The minimizer $\mathbf{Z}$ of \eqref{gh_p} is rank 16. The figure in the right shows a \emph{soft assignment} between $X$ and $Y$ obtained from $\mathbf Z$ by computing $\hat{\mathbf{Z}}\mathbf{1}$ and rearranging accordingly.}
\end{figure}



\begin{theorem} \label{triangle}
Consider $\dpp$ and $\dinf$ defined in equations~\eqref{gh_p} and \eqref{gh_inf} respectively for $\mathcal A= \mathcal{GH},\operatorname{Reg},\operatorname{Sur}$. Then we have:
\begin{enumerate}
\item For $X,Y,W \in \mathcal G_n $, and $1\leq p< \infty$, $\dpp$  satisfies the triangle inequality. \label{triangle_1}
\item For $X,Y,W \in \mathcal G_{<\infty}$, $\dinf$ satisfies the triangle inequality. \label{triangle_2}
\end{enumerate}
\end{theorem}
\begin{proof}
We begin by proving part (\ref{triangle_1}).  Note that it suffices to
show that for $p\geq 1$, 
\begin{equation} \label{triangle_p}
\dpp(X,W)^p\leq \dpp(X,Y)^p+ \dpp(Y,W)^p.
\end{equation}
This follows from the fact that for $a,b>0$ and $p\geq 1$ we have
$a^p+b^p\leq (a+b)^p$ and therefore if equation~\eqref{triangle_p} holds we
have: 
$$\dpp(X,W)\leq \sqrt[p]{\dpp(X,Y)^p+ \dpp(Y,W)^p}\leq \dpp(X,Y)+ \dpp(Y,W).$$

Now let $\mathbf Z$ and $\mathbf V$ the minimizers in equation~\eqref{gh_p} for $X,Y$ and $Y,W$ respectively in $\mathcal A$. From $\mathbf Z$ and $\mathbf V$ we construct $\mathbf T$ feasible for $X,W$ in equation~\eqref{gh_p} and we show the objective function in $\mathbf T$ is smaller or equal to $ \dpp(X,Y)+ \dpp(Y,W)$.

If $x_i, x_{i'} \in X$, $y_j, y_{j'} \in Y$ and $w_k, w_{k'} \in W$ let $\mathbf T$ the unique feasible matrix in $\mathcal A$ that satisfies
\begin{equation} \hat{\mathbf T}_{ik,i'k'}\colon=\sum_{j,j'}\hat{\mathbf Z}_{ij,i'j'}\hat{\mathbf V}_{jk,j'k'}. \label{T}
\end{equation}

To see that $\mathbf T$ is well-defined, observe that it is
straightforward to check that $\mathbf T$ satisfies the linear and
inequality constraints of $\mathcal A$ using the fact that $\mathbf Z$
and $\mathbf V$ belong to $\mathcal A$. In order to verify that
$\mathbf T$ is positive semidefinite, consider the Cholesky
decompositions of $\mathbf Z$ and $\mathbf V$.  Then
$$\hat {\mathbf Z}_{ij,i'j'}= z_{ij}^\top  z_{i'j'},\quad \hat{ \mathbf V}_{jk,j'k'}=  v_{jk}^\top  v_{j'k'}$$ 
where $z$ and $v$ do not necessarily correspond to the last column in equation~\eqref{Z}. In fact $z$ is a $r\times n^2$ matrix where $r$ is the rank of $\hat {\mathbf Z}$ and $z_{ij}$ is the column of $z$ indexed by $i=1,\ldots, n$ and $j=1,\ldots, n$. Then note
$$\hat {\mathbf T}_{ik,i'k'}=\sum_{j,j'}\hat{\mathbf Z}_{ij,i'j'}\hat{\mathbf V}_{jk,j'k'} = \left\langle \sum_{j}z_{ij} \otimes v_{jk}, \sum_{j'} z_{i'j'}\otimes v_{j'k'} \right\rangle $$
therefore $\hat{\mathbf T}$ is PSD since it is a Gram matrix, and $\mathbf T$ is PSD since it has the same rank as $\hat{\mathbf T}$.

For the triangle inequality we need to show
\begin{multline} \label{triangle}
\sum_{i,i'}\sum_{k,k'}\mathbf T_{ik,i'k'} |d_X(x_i,x_{i'})- d_W(w_k, w_{k'})|\leq \sum_{i,i'}\sum_{j,j'} \mathbf Z_{ij,i'j'} |d_X(x_i,x_{i'})- d_Y(y_j, y_{j'})|  \\ +\sum_{j,j'}\sum_{k,k'} \mathbf V_{jk,j'k'} |d_X(x_j,x_{j'})- d_W(w_k, w_{k'})|.
\end{multline}
In the case we are dealing with, where $|X|=|Y|=|W|$, the constraints 
\[
\sum_{i,i'} \mathbf Z_{ij,i'j'}\geq1 \qquad\textrm{and}\qquad
\sum_{k,k'}\mathbf V_{jk,j'k'}\geq1
\]
are tight, meaning that equality holds, so for all $j,j'$, we can
multiply by 1 and rewrite the RHS of equation~\eqref{triangle} as 
\begin{eqnarray*}
\sum_{i,i'}\sum_{j,j'} \mathbf Z_{ij,i'j'} |d_X(x_i,x_{i'})- d_Y(y_j, y_{j'})| \sum_{k,k'}\mathbf V_{jk,j'k'} + \sum_{j,j'}\sum_{k,k'} \mathbf V_{jk,j'k'} |d_X(x_j,x_{j'})- d_W(w_k, w_{k'})| \sum_{i,i'} \mathbf Z_{ij,i'j'}\\
= 
\sum_{i,i'} \sum_{k,k'} \sum_{j,j'} \mathbf Z_{ij,i'j'} \mathbf V_{jk,j'k'} ( |d_X(x_i,x_{i'})- d_Y(y_j, y_{j'})| +  |d_Y(y_j,y_{j'})- d_W(w_k, w_{k'})|) 
\end{eqnarray*}
Now it is clear that equation~\eqref{triangle} follows from the
triangle inequality in $\mathbb R$, which completes the verification
of part~\eqref{triangle_1}.

In order to prove part~\eqref{triangle_2}, now we let $X,Y,Z$ to be finite metric spaces with arbitrary number of points. And we let $\mathbf Z$ and $\mathbf V$ the minimizers in equations~\eqref{gh_inf} for $X,Y$ and $Y,W$ respectively as before. We define $\mathbf T$ as in equation~\eqref{T}. We know $\mathbf T$ is feasible for equation~\eqref{gh_inf} so the remaining step to prove is 
\begin{equation} \max_{\mathbf T\neq 0} \Gamma_{ik,i'k'}\leq \max_{\mathbf Z\neq 0} \Gamma_{ij,i'j'} + \max_{\mathbf V\neq 0} \Gamma_{jk,j'k'}
\label{inf_ineq}
\end{equation} 
Let $(ik,i'k')$ the $\arg\max$ of the left hand side of equation~\eqref{inf_ineq}. Since ${\mathbf T}_{ik,i'k'}\neq 0$ and ${\mathbf T}_{ik,i'k'}=\sum_{j,j'}{\mathbf Z}_{ij,i'j'} {\mathbf T}_{kj,k'j'}$ then there exists $j,j'$ such that ${\mathbf Z}_{ij,i'j'}\neq 0$ and ${\mathbf T}_{kj,k'j'}\neq 0$.
Then we have 
\begin{multline} 
\dinf(X,W) \leq \max_{\mathbf T\neq 0} \Gamma_{ik,i'k'} = |d_X(x_i, x_{i'})- d_{W}(w_k, w_{k'})| \\
\leq |d_X(x_i, x_{i'})- d_Y(y_j, y_{j'})| + |d_Y(y_j, y_{j'})- d_W(w_k, w_{k'})| \leq \max_{\mathbf Z\neq 0} \Gamma_{ij,i'j'} + \max_{\mathbf V\neq 0} \Gamma_{jk,j'k'} \\= \dinf(X,Y)+\dinf(Y,W).
\end{multline}
\end{proof}

\begin{remark} The same argument will show that $\tilde d_{\mathcal{GW},p}$ satisfies triangle inequality as long as we add the constraint $\hat{\mathbf Z}\mathbf 1 = n\mathbf z$, where $n=|X|=|Y|=|W|$ and the measure of each of the points is equal $1/n$.  
\end{remark}

\subsection{Monotonicity and continuity properties}

The following lemma shows the monotonicity of $\dpp$ with respect to $p$. 
The second part of the lemma proves continuity of $\dpp$ at infinity. 
 
\begin{proposition} \label{bound}
For any $X$,$Y$ finite metric spaces we have:
\begin{enumerate}
\item If  $1\leq p\leq q< \infty $ then $\dpp(X,Y)\leq \dqq(X,Y) \leq \dinf(X,Y)$. \label{1}
\item $\lim_{ p\to \infty }\dpp(X,Y) = \displaystyle\min_{\mathbf Z \in \mathcal{ \tilde T}} \max_{i,j,i',j'\colon \mathbf Z \neq 0} \Gamma_{ij,i'j'} =\dinf(X,Y)$. \label{2}
\end{enumerate}
\end{proposition}

\begin{proof}
Let $\mathbf Z \in  {\mathcal A}$ optimal for equations~\eqref{gh_inf} or \eqref{gh_p}  for some value of $p$. Then $ \mathbf{ 1^\top \hat{ Z} 1}= n^2$. The weighted power mean inequality implies 
$$\left( \frac{1}{n^2} \sum_{ij,i'j'} \Gamma_{ij,i'j'}^p \mathbf Z_{ij,i'j'} \right)^{1/p} \leq \left( \frac{1}{n^2} \sum_{ij,i'j'} \Gamma_{ij,i'j'}^q \mathbf Z_{ij,i'j'} \right)^{1/q} \leq  \max_{i,j,i',j'\colon \mathbf Z_{ij,i'j'}\neq 0} \Gamma_{ij,i'j'}$$
and taking the infimum in $\mathbf Z$ we obtain \eqref{1}. 

Now for fixed $\mathbf Z$ let $\Gamma_{\mathbf Z}^*=\max \{ \Gamma_{ij,i'j'}\colon \mathbf Z_{ij,i'j'}\neq 0\}$, then using the standard calculus argument 
$$\lim_{p\to \infty} \left( \frac{1}{n^2} \sum_{ij,i'j'} \Gamma_{ij,i'j'}^p \mathbf Z_{ij,i'j'} \right)^{1/p} = \Gamma^*_{\mathbf Z} \lim_{p\to \infty } \left( \sum_{ij,i'j'} \left(\frac{\Gamma_{ij,i'j'}}{\Gamma^*_{\mathbf Z}}\right)^p \frac{\mathbf Z_{ij,i'j'}}{n^2} \right)^{1/p} =\Gamma^*_{\mathbf Z}.$$
and taking infimum in $\mathbf Z$ we obtain \eqref{2}. 
\end{proof}

Proposition \ref{bound} holds for metric spaces $X$ and $Y$ with possibly different number of points and it says that even though $\dpp$ may not satisfy the triangle inequality, it does in the limit $p\to \infty$.


\subsection{Extension of the distance to compact infinite sets} \label{extension}

Every compact metric space $X$ is the limit of a sequence of finite
metric spaces in the Gromov-Hausdorff topology, denoted here by
$\tau_{\mathcal{GH}}$ (see for instance~\cite[Example 7.4.9]{bbi}). In
fact, by taking $\epsilon_n\to 0$ and choosing a finite
$\epsilon_n$-net $S_n$ in $X$ for every $n$, we get
$S_n\stackrel{\mathcal{GH}}{\to} X$ because 
$$ \d(X,Sn)\leq d_{\mathcal{H}}(X,S_n)\leq \epsilon_n.$$

This property inspires the following definition of an actual distance between compact metric spaces.

\begin{definition} \label{d_hat}
Let $X,Y$ compact metric spaces. Given $\epsilon_n\to 0$, let $X_n, Y_n$ respective $\epsilon_n$-nets of $X$ and $Y$, with the same number of points $N$. Define 
\begin{equation} \label{infinite_d}
\hat d_{\mathcal A,p}(X,Y)=\inf_{\epsilon_n,\;X_n\,Y_n}\lim \sup_{n\to \infty} \dpp(X_n,Y_n) 
\end{equation}
\end{definition}

Note that  $\lim\sup$ exists because $\dpp(X_n,Y_n)\leq \dinf(X_n,Y_n)\leq \frac12\max(\operatorname{diam}(X), \operatorname{diam}(Y))$ for all $n$. Also, note the triangle inequality holds for this limit, which also implies that $\hat d_{\mathcal A,p}$ and $\dpp$ may not agree.


To illustrate how the right hand side of \eqref{infinite_d} behaves, let's say that $|X|<|Y|$ and for some $n$ the $\epsilon_n$-net $Y_n$ of $Y$ has at least $N$ points and $|X|<N$, then consider $X_n$ to be $X$ with some repeated points and run the SDP \eqref{gh_p} or \eqref{gh_inf} to compute $\dpp(X_n,Y_n)$ so that $|Y_n|=|X_n|=N$. Note that this is well defined and when $\dpp(X_n,X)$ exists (i.e. $\mathcal A=\mathcal{GH},\operatorname{Sur}$) we have $\dpp(X_n,X)=0$ because the matrix $\mathbf Z$ corresponding to the surjective function $X_n\to X$ is in $\mathcal A$ and has objective value 0.


\subsection{Comparison with the Gromov-Hausdorff distance}
Let $\mathcal X$ denote the set of isometry classes of compact metric spaces.
Definition~\ref{d_hat} extends the relaxed distances \eqref{gh_p} and
\eqref{gh_inf} to $\mathcal X$, obtaining the function $\hat
d_{\mathcal{A},p}\colon \mathcal X \times \mathcal X \to \mathbb R$.

\begin{lemma} For $X,Y \in \mathcal X$ we have for $1\leq p\leq \infty$
$$\hat d_{A,p}(X,Y)\leq \d(X,Y)$$
\end{lemma}
\begin{proof}
First assume $X,Y$ are finite and let $R \in \mathcal R (X,Y)$ the minimizer in \eqref{gh_cont}. If $|R|=N$ let $X_N, Y_N$ the $\epsilon$-net so that every element of $X$ appears in $X_N$ as many times as it appears in $R$ (and the same for $Y_N$). Then the bijective function between $X_N$ and $Y_N$ corresponding to $R$ induces a feasible $\mathbf Z$, proving the result in the finite case. For the remaining case consider a $\epsilon_n$-net where $\epsilon_n\to 0$.  
\end{proof}

Now consider $X,Y$ finite metric spaces. First observe that $\mathcal A=\mathcal{GH}$ includes the $\mathbf Z$ induced by all elements in $\mathcal R(X,Y)$, which  together with Proposition \ref{bound} implies
$$\tilde d_{\mathcal{GH},p}(X,Y) \leq \tilde d_{\mathcal{GH},\infty}(X,Y)\leq \d(X,Y).$$ 
Since $\operatorname{Sur} \subset \mathcal GH$ we have 
$$ \tilde d_{\mathcal{GH},p}(X,Y) \leq \tilde d_{\operatorname{Sur},p}(X,Y),$$
and if $|X|=|Y|$ we can consider that $\operatorname{Reg} \subset \operatorname{Sur}$ therefore
$$ \tilde d_{\mathcal{GH},p}(X,Y) \leq \tilde d_{\operatorname{Sur},p}(X,Y) \leq \tilde d_{\operatorname{Reg},p}(X,Y).$$
Also, the smaller the set $\mathcal A$, the more likely is the relaxation to produce a tight solution (a rank 1 solution, corresponding with an element of $\mathcal R(X,Y)$).
Note that neither $\tilde d_{\operatorname{Sur},p}$ nor  $\tilde d_{\operatorname{Reg},p}$ are comparable with $\d$.


\subsection{Topologies induced by relaxed distances}

Any metric or pseudometric $d$ defines a topology $\tau$ characterized
by the property that given a sequence $X_n$, we have convergence
$X_n\stackrel{\tau}{\to} X$ if and only if $d(X_n,X)\to 0$.  In
particular, the Gromov-Hausdorff distance induces a topology on the
set of isometry classes of compact metric spaces.  The
Gromov-Hausdorff topology is a fairly weak topology; for example,
there are many compact sets. Proposition 7.4.12 in \cite{bbi}
characterizes the Gromov-Hausdorff convergence in terms of
$\epsilon$-nets, implying that if a sequence $\{X_n\}$ converges in
the Gromov-Hausdorff topology, then for all $\epsilon>0$, the
cardinality of $\epsilon$-nets is uniformly bounded over
$X_n$. Therefore if a class $\mathcal X$ of metric spaces is
pre-compact (i.e. any sequence of elements of $\mathcal X$ has a
convergent subsequence) in the Gromov-Hausdorff topology, then for
every $\epsilon > 0$ the size of a minimal $\epsilon$-net is uniformly
bounded over all elements of $\mathcal X$.  The analysis in \cite{bbi}
shows that this property of $\mathcal X$, along with the fact that the
diameters of its members are uniformly bounded (what is called totally
boundedness, Definition \ref{7.4.14}), is sufficient for
pre-compactness (Theorem \ref{7.4.15}).  

Let $\hat\tau_{\mathcal A, p}$ ($1\leq p \leq \infty$), and $\tilde
\tau_{\mathcal A,\infty}$ denote the topologies induced by the
pseudometrics $\hat d_{\mathcal A, p}$ ($1\leq p\leq \infty$) and
$\tilde d_{\mathcal A, \infty}$, respectively.  We obtain an analogous
characterization of pre-compact sets in the topology for
$\hat{\tau}_{\mathcal A, p}$ for $1\leq p \leq \infty$ in
Corollary~\ref{cor} below. 




\begin{proposition}{\cite[7.4.12]{bbi}} \label{7.4.12}
For compact metric spaces $X$ and $\{X_n\}_{n=1}^\infty$, $X_n\stackrel{\tau_{\mathcal{GH}}}{\longrightarrow}X$ if and only if the following holds. For every $\epsilon>0$ there exist a finite
$\epsilon$-net $S$ in $X$ and an $\epsilon$-net $S_n$ in each $X_n$ such that $S_n \stackrel{\tau_{\mathcal{GH}}}{\longrightarrow}S$.
Moreover these $\epsilon$-nets can be chosen so that, for all sufficiently large $n$, $S_n$ have the same cardinality as $S$.
\end{proposition}

Note that by construction (Definition \ref{d_hat}) the characterization of convergence by $\epsilon$-nets from Proposition~\ref{7.4.12} is also true when substituting $\tau_{\mathcal{GH}}$ by $\hat{\tau}_{\mathcal A, p}$, $1\leq p \leq \infty$.

\begin{definition}{\cite[7.4.14]{bbi}} \label{7.4.14} A class $\mathcal X$ of compact metric spaces is totally bounded if 
\begin{enumerate}
\item There exists a constant $D$ such that for all $X\in \mathcal X$, $\operatorname{diam}(X)<D$.
\item  For every $\epsilon > 0$ there exists a number $N_\epsilon$  such that
every $X \in \mathcal X$ contains an $\epsilon$-net consisting of at most $N_\epsilon$ points.
\end{enumerate}
\end{definition}

\begin{theorem}{\cite[7.4.15]{bbi}} \label{7.4.15}
Any uniformly totally bounded class $\mathcal X$ of compact metric spaces is pre-compact in the Gromov-Hausdorff topology. 
\end{theorem}

By Theorem~\ref{7.4.15}, we know that if $\mathcal X$ is totally
bounded and $\{X_n\}_{n=1}^\infty$ is a sequence in $\mathcal X$ then
it contains a convergent subsequence in $\mathcal X$. Since
$\hat{d}_{\mathcal A, p} \leq \d$, that subsequence is also convergent
in $\hat{\tau}_{\mathcal A,p}$, which immediately implies the
following corollary: 

\begin{corollary} \label{cor}
Any uniformly totally bounded class $\mathcal X$ of compact metric spaces is pre-compact in the topology $\hat\tau_{\mathcal A, p}$ for $1\leq p \leq \infty$. 
\end{corollary}

\subsection{Local topological properties}

In the space of compact metric spaces we know that $\d(X,Y)=0$ if and
only if $X$ and $Y$ are isometric. The example at the beginning of
Section \ref{sec:semimetric} shows that this is not true for $\dpp$ in
general. However, in this section we show it is true for \emph{most}
finite $X$.

\begin{definition}
Let $X$ a finite metric space. We say that $X$ is generic if $X\in \mathcal G_{<\infty}$ and all pairwise distances in $X$ are different and non-zero. 
\end{definition}

The name generic is justified in the following sense: 
if $X\in \mathcal G_n$ is not generic, for all $\epsilon>0$ there exists $Y\in \mathcal G_n$ such that $\d(X,Y)<\epsilon$ and $Y$ is generic. Also, if $X\in \mathcal G_n$ is generic there exists $\epsilon>0$ such that for all $Y\in \mathcal G_n$ that satisfy$\d(X,Y)<\epsilon$ we have that $Y$ is also generic. Which proves the following remark:
\begin{remark}
The set of generic metric spaces is dense in $\tau_{\mathcal{GH}}|_{\mathcal G_{<\infty}}$ 
and open in $\tau_{\mathcal{GH}}|_{\mathcal G_n}$.
\end{remark}

\begin{lemma} If $X$ and $Y$ are generic and  $\dpp(X,Y)=0$ then $X$ and $Y$ are isometric. 
\end{lemma}

\begin{proof}
Assume without loss of generality $|X|\geq|Y|$ and $\dpp(X,Y)=0$. Let $\mathbf Z$ the solution of \eqref{gh_p} for $X,Y$ with objective value 0. The constraint $\sum_{j,j'} \mathbf Z_{ij,i'j'}=1$ for all $i,i'$ implies that, given $i\neq i'$ there exists $j,j'$ such that $\mathbf Z_{ij,i'j'}>0$. Since the objective value is 0, that implies that $d_{X}(x_i,x_{i'})=d_Y(y_j,y_{j'})$. Since all pairwise distances in $X$ are different, that completely determines all distances in $Y$ and in particular it implies $|X|=|Y|$, $X$ and $Y$ are isometric, and the unique solution of $\eqref{gh_p}$ corresponds to the isometry between $X$ and $Y$. 
\end{proof}

\begin{corollary}
If $X\in \mathcal G_n$ is generic there exists a neighborhood of $X$ in $\mathcal G_{n}$ such that for all $Y$ in that neighborhood $$\d(X,Y)=\frac{1}{2}\max_{i,j=1\ldots n}|d_{X}(x_i,x_j)-d_{Y}(y_i,y_j)|$$ 
($Y$ is a small enough perturbation of a metric space isometric with $X$ where we label the points such that the isometry is $x_i\mapsto y_i$ for all $i$). In particular we can think of the neighborhood where that property holds as the neighborhood of $X$ with radius $\Delta/n$ where $\Delta$ is the smallest non-zero entry of the matrix $\Gamma(X,X)$. In this setting we have that $$\dpp(X,Y)=\frac12 \left( \frac1{n^2} \sum_{i,j}|d_{X}(x_i,x_j)-d_{Y}(y_i,y_j)|^p  \right)^{1/p}$$
in the neighborhood of $X$ of radius $\Delta^p/n$. And, in the neighborhood of $X$ of radius $\Delta$ we have
$$\dinf(X,Y)=\d(X,Y).$$

This implies that the topologies $\tau_{\mathcal{GH}}|_{\mathcal G_n}$ and $\tilde\tau_{\mathcal{GH},p}|_{\mathcal G_n}$ are equivalent for all finite $n$ and $p$.
And we have $\dinf$ and $\d$ are generically locally the same.
 
\end{corollary}




 

  

\section{GHMatch: a rank-1 augmented Lagrangian approach towards the registration problem}

In the previous sections, we have studied an approach to the problem
of computing the Gromov-Hausdorff distance (equation~\eqref{gh}) via
semidefinite optimization.  Here we first lift the variable $\mu \in
\mathbb R^{nm}$ to a symmetric variable $\mathbf Z \in \mathbb R^{nm+1
  \times nm+1}$ such that $\operatorname{rank}(\mathbf Z)=1$.  We then
relax the non-convex rank constraint to the convex constraint $\mathbf
Z\succeq 0$.  

There are many attractive properties of the semidefinite
relaxations.  For one thing, there are many software packages that
efficiently provide global solutions to semidefinite programs (e.g.,
SDPNAL+~\cite{yang2015sdpnal+}).  Moreover, there is a great deal of
research energy directed at producing more efficient SDP solvers; the
field is rapidly evolving and solvers are getting more efficient every
day.  Furthermore, SDP relaxations have the advantage of often being
tight: in our situation, we have observed numerically that the
solution $\mathbf Z$ frequently has rank 1.  In this case, the
semidefinite optimization finds the global solution of the original
problem, and also provides a certificate of its optimality.  This
property has recently been exploited to efficiently produce
certificates of optimality of solutions found by fast non-convex
algorithms that typically may converge to local
optima~\cite{Bandeira2015note, pcc}.

On the other hand, the semidefinite relaxations have the disadvantage
that they square the number of variables of the original problem: even
with efficient solvers, this expansion makes these problems
intractable for large sets of points.  Also, when the SDP is not
tight, it may produce a high rank solution $\mathbf Z$ that may not be
easily rounded to a feasible $\mu$.   

Motivated by these concerns, in this section we propose a non-convex
optimization approach for the registration problem.  Here we trade the
global optimality guarantee and the pseudometric the semidefinite
optimization provides for computational efficiency and a guarantee of
a true (albeit not necessarily globally optimal) correspondence.  
 
We will  assume that $|X|=|Y|=n$. By restricting equation~\eqref{gh}
to this case and replacing the infinity norm by the $p$-norm
formulation  we obtain the following non-convex optimization problem,
where $\mathbf{y}\in \mathbb R^{n^2}$ is indexed by a pair of
variables $ij$ where $i,j=1,\ldots, n$. 
\begin{equation} 
\min_{\mathbf y}\langle \Gamma^{(p)}, \mathbf{yy}^\top \rangle \quad\text{ subject to } \sum_{i=1}^n \mathbf y_{ij}=1, \; \sum_{j=1}^n \mathbf y_{ij}=1, \; 0\leq \mathbf y_{ij} \leq 1 \label{non-convex}
\end{equation}

Now, instead of considering a semidefinite relaxation as we did
previously, we propose a greedy method to directly solve \eqref{non-convex}. 
Let $A\in \mathbb R^{2n\times n^2}$ such that $A\mathbf y=\mathbf 1$
if and only if $\sum_{i=1}^n \mathbf y_{ij}=1$ and $ \sum_{j=1}^n
\mathbf y_{ij}=1$ and let $b=\mathbf{1}\in \mathbb R^{2n}$.  Then
equation~\eqref{non-convex} is equivalent to the following quadratic
optimization problem with linear and box constraints.  
\begin{equation}
\min_{\mathbf y}\mathbf y ^\top  \Gamma^{(p)}\mathbf y \quad \text{ subject to } A\mathbf y= b, \; 0\leq \mathbf y \leq 1 \label{quadratic}
\end{equation}

In order to solve problem~\eqref{quadratic}, we use a projected
augmented Lagrangian approach (e.g.,
see~\cite[Algorithm~17.4]{nocedal}).

\begin{equation}
\mathcal L(\mathbf y, \lambda, \sigma) = \mathbf y ^\top  \Gamma^{(p)}\mathbf y - \lambda^T(A\mathbf y -b) + \frac\sigma2 \|A\mathbf y - b\|_2^2
\end{equation} 

We propose the algorithm GHMatch (see Algorithm \ref{algorithm}).
Theoretical convergence analysis for GHMatch is left for future
work.  In the next section, we describe numerical experiments that
indicate the performance of the algorithm.  In the experiments we
conducted, we found that this procedure converges to a local minimum
of equation~\eqref{non-convex}.  That solution may be rounded to an
actual correspondence between the point sets $\mathbf {\bar y}$, and
therefore the value $\langle \Gamma, \mathbf{\bar y} \mathbf{\bar
  y^\top} \rangle$ is an upper bound for the Gromov-Hausdorff
distance.

\begin{algorithm}
\begin{algorithmic}[1]
\State {$\mathbf y_0 \gets \frac{1}{n}\mathbf{1} \in \mathbb R^{n^2}$}
\State $\lambda_0\gets \mathbf{1} \in \mathbb R^{2n}$
\State $\sigma_0 \gets 5$
\State $\mu \gets 10$ 
    \For {$k = 0,1,2,\ldots $}
        \State $\mathbf y \gets \arg\min_{0 \leq \mathbf y\leq 1} \mathcal L(\mathbf y, \lambda_k, \sigma_k)$ \Comment{Use $\mathbf y_k$ as initial point for this minimization}
        \State $\lambda_{k+1} \gets \lambda_k -\sigma_{k}(A\mathbf y_{k+1}-b)$
        \State $\sigma_{k+1}\gets \mu \sigma_k$
    \EndFor 
    \For {$i=1,\ldots,n$} \Comment{To find the map corresponding with $\mathbf y_{T}$}
    	\State $\operatorname{map}(i)=\arg\max_{j=1,\ldots, n} \mathbf y_{T} (1+(i-1)n \colon in)$
    \EndFor
\end{algorithmic}
\caption{GHMatch\label{algorithm}}
\end{algorithm}


\section{Numerical performance}

In this section, we describe the results of a number of numerical
experiments to explore the applications of our new distance and the
performance of our augmented Lagrangian approach.

\begin{figure}
\vspace{-2.5 cm}
\includegraphics[width=0.45\textwidth]{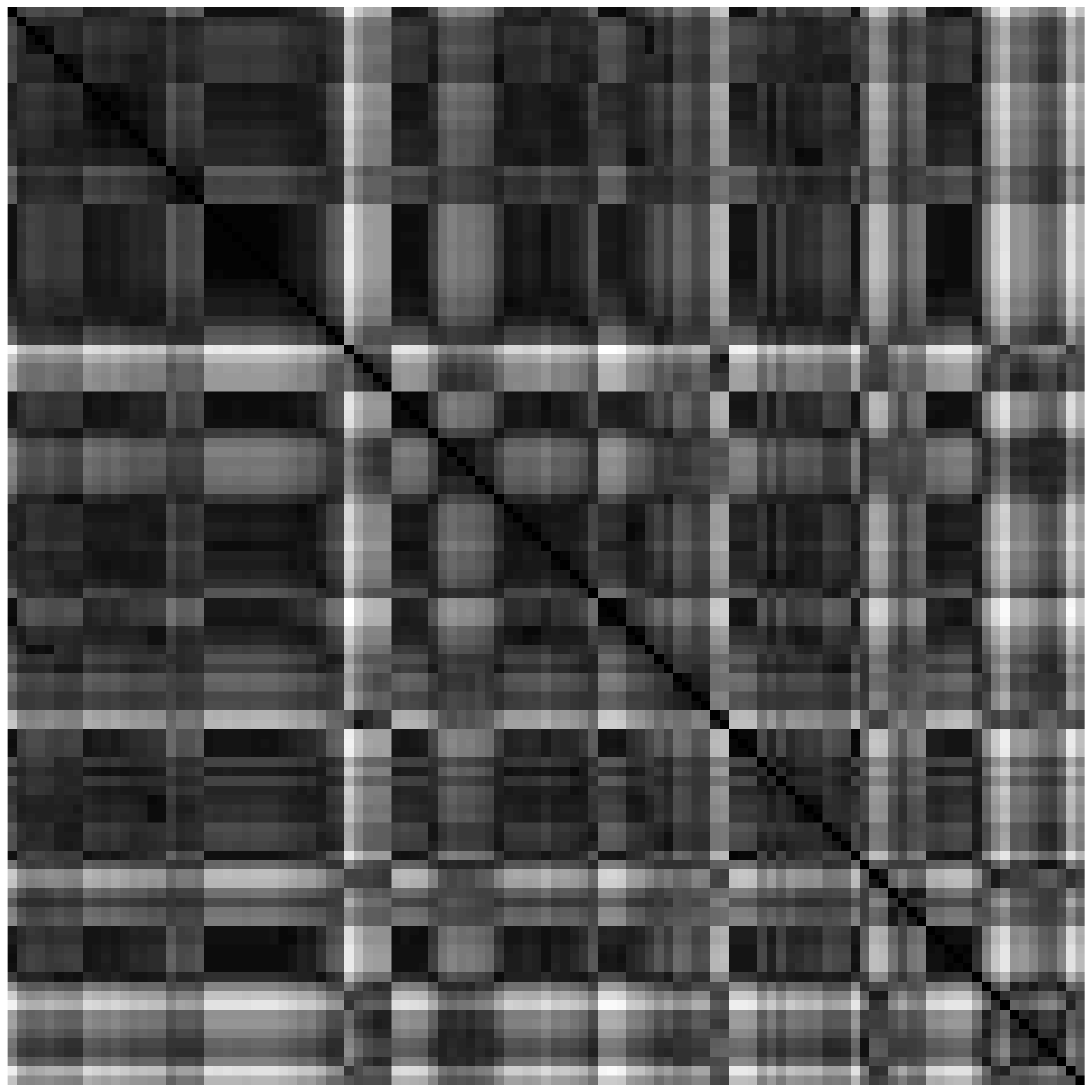}
\includegraphics[width=0.45\textwidth]{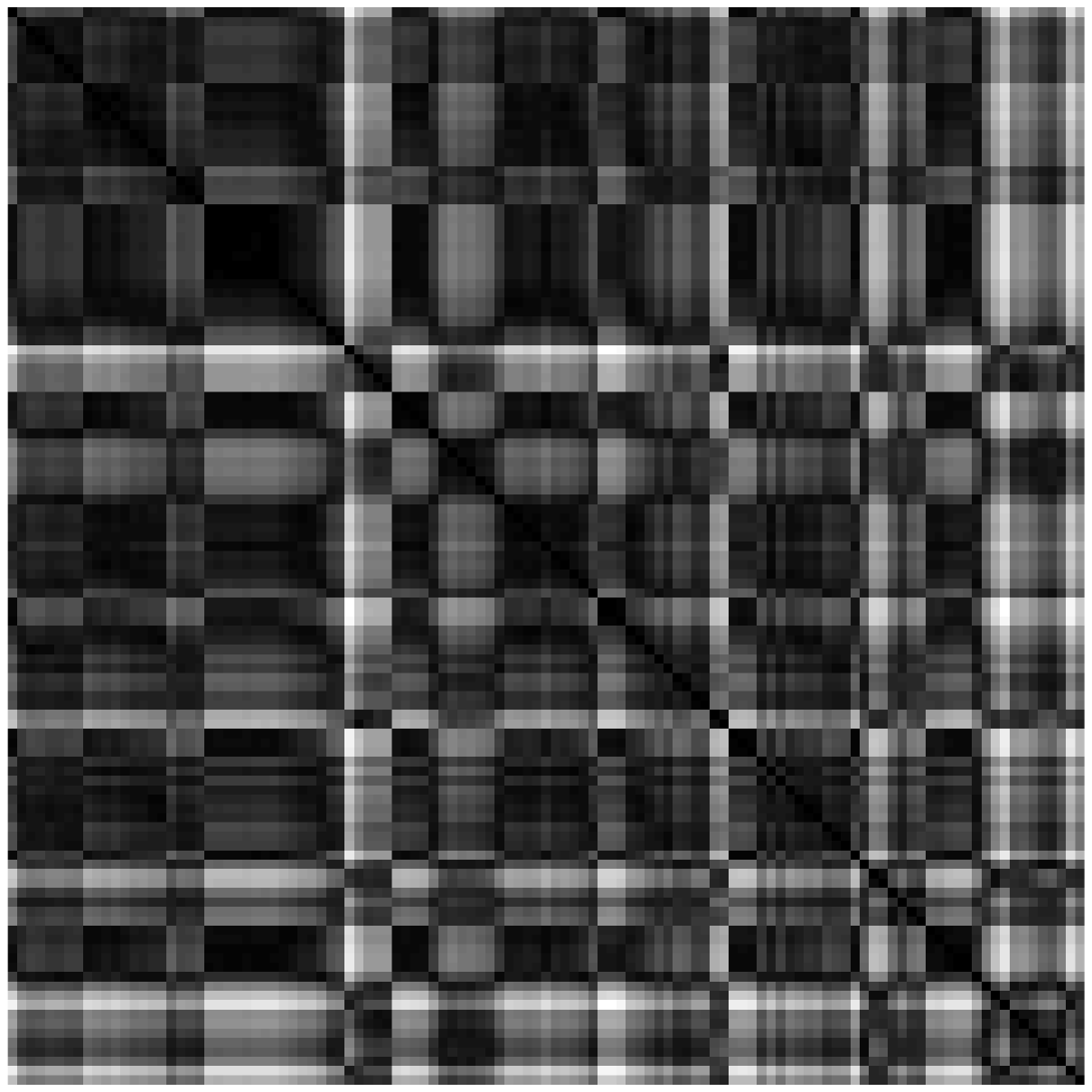}
\vspace{-2.5cm}
\caption{\label{distances} (Left) Distance matrix obtained from computing the best rigid transformation that maps the corresponding labeled landmarks from the teeth dataset described in Section~\ref{sec:teeth}. (Right) Distance matrix obtained from computing the SDP distance $\tilde d_{\mathcal{GH}, p}$ with $p=1$. Darker color corresponds to smaller distance. We observe the same distance patterns in both matrices even though the scales are different.}
\end{figure}

\subsection{Classification using the distance $\tilde d_{\mathcal
    {GH}}$}\label{sec:teeth}
In order to validate our distance numerically we compare with the
numerical experiments described in~\cite{Boyer}, using data and
algorithms available on Yaron Lipman's personal website~\cite{Lipman-website}.  As we
describe below, we find that our procedure produces results that are
competitive with this procedure.

In \cite{Boyer} the authors propose an algorithm to automatically
quantify the geometric similarity of anatomical surfaces based on a
distance inspired by the Gromov-Wasserstein distance which is
invariant under conformal maps.  They experiment with a real dataset
coming from surfaces of teeth of different species; they compare the
results of their algorithm with a method based on an expert selecting
16 landmarks on each tooth and then finding 
the best rigid
transformation to match the labeled landmarks.
Specifically, they work with a dataset consisting of 116 teeth. For
each tooth, they find the closest tooth according to each distance,
and then see whether they are in the same category.

We perform the same experiment on 115 of the teeth (since one of them
seems to be in a different scale), but without providing our algorithm
with the correspondence between the landmarks.
To be precise, we consider the metric spaces 
\[
X_i=\{p^i_1, \ldots p^i_{16}\}, \quad  i=1\ldots 115.
\]
The points of these metric spaces are
the landmarks chosen by the expert, and the metric is given by the euclidean distance between the landmarks.
We compare the distance matrix $\tilde d_{\mathcal {GH}, p}(X_i,X_j)$
with the distance obtained by the software from \cite{Lipman-website} that finds the best rigid transformation
that sends the $n$-th point of $X_i$ to the $n$-th point of $X_j$ for
$n=1,\ldots 16$. 
See Figure~\ref{distances} for a visual
comparison of the distance matrices. 

When running the nearest-neighbor classification test as described
above, we obtain very similar performance: $0.85$ frequency of success in our
distance against $0.91$ for the conformal Wasserstein distance proposed in \cite{Boyer} and $0.92$ for the landmark comparison algorithm that uses the a priori known correspondence.
We find this result very encouraging given that our algorithm does not
make any geometric assumptions about the teeth (e.g., we do not assume
they are smooth surfaces), in contrast to~\cite{Boyer}.

\subsection{Performance of GHMatch}
\begin{figure}
\includegraphics[width=0.8\textwidth, trim={0 1.5cm 0 .5cm},clip]{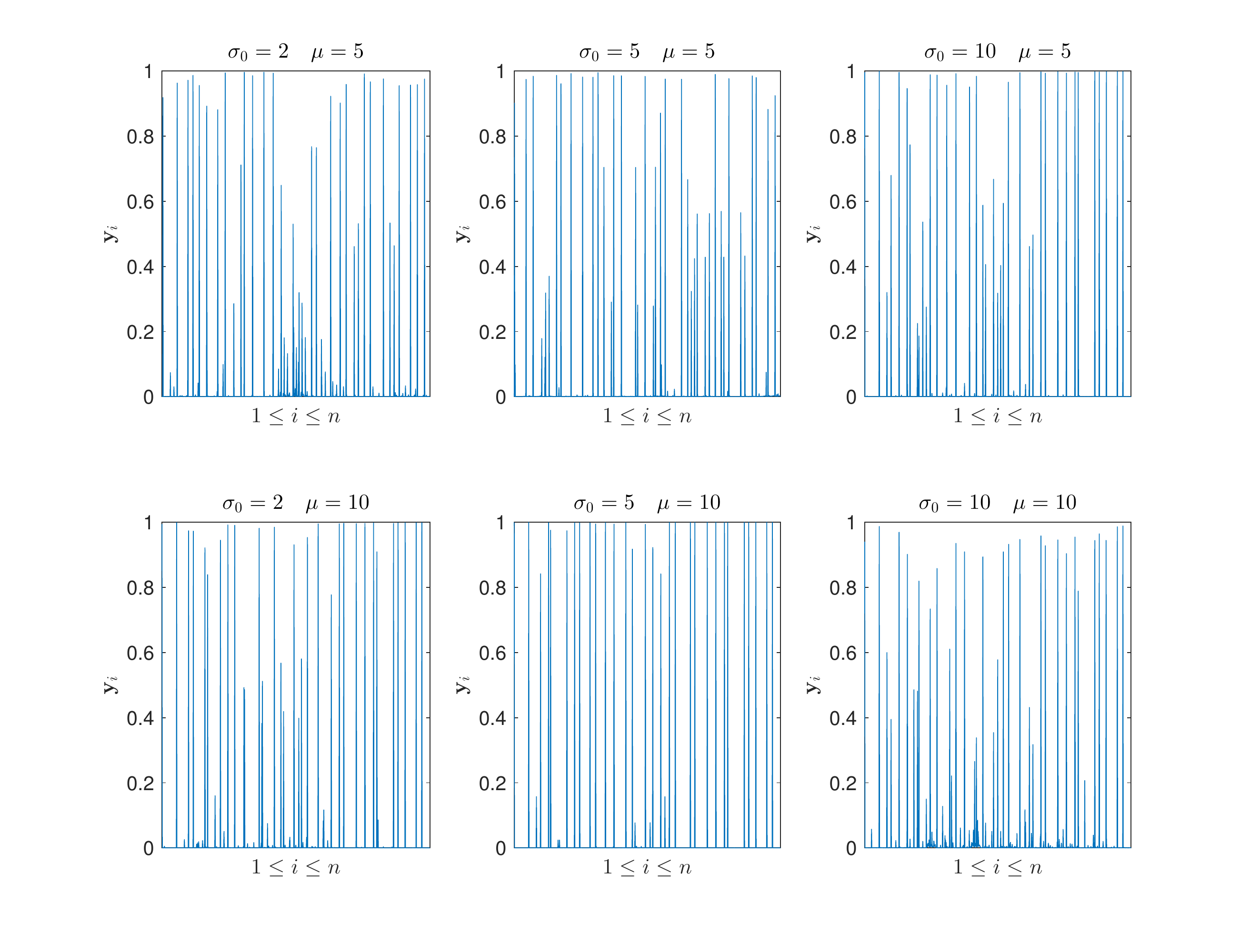}
\caption{This is the value of convergence of $\mathbf y_k$ for
  different parameters $\sigma_0$ and $\mu$, for $n=30$. All these
  $\mathbf y$ satisfy the linear constraint $A\mathbf y=b$ but the
  choice of the parameters determine how close the vector $\mathbf y$
  is to a feasible vector with entries in
  $\{0,1\}$.\label{error_plots}} 
\end{figure}

\begin{figure}
\vspace{-1 cm}
\includegraphics[height=0.32\textwidth]{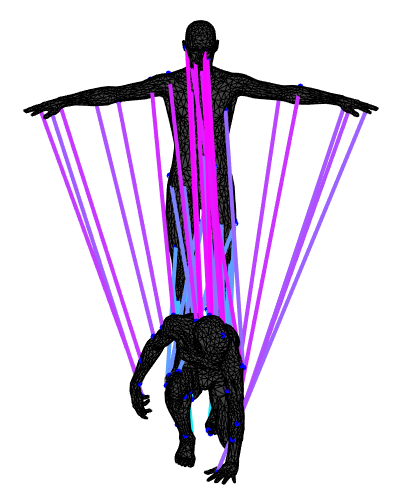}
\includegraphics[height=0.32\textwidth]{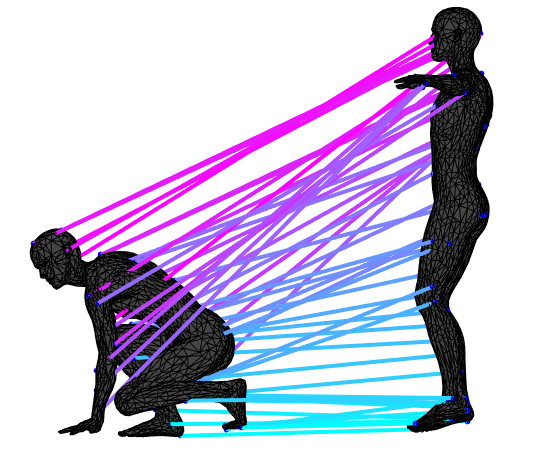}
\includegraphics[height=0.32\textwidth]{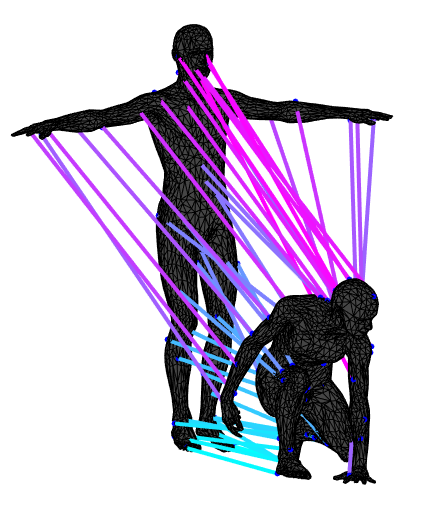}
\includegraphics[width=0.45\textwidth, trim={5cm 3.2cm 5cm 2cm},clip]{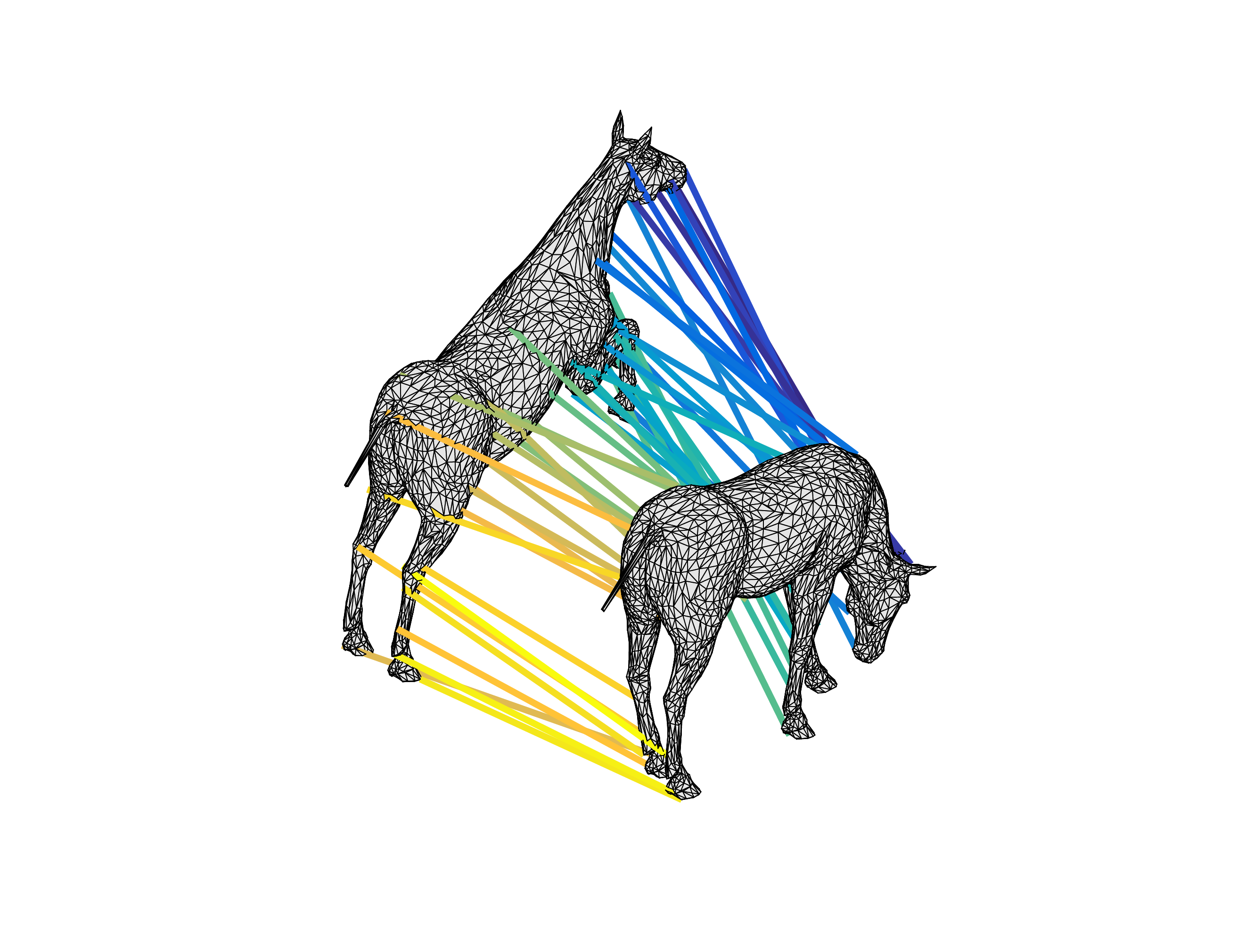}
\includegraphics[width=0.45\textwidth, trim={5cm 3.2cm 5cm 2cm},clip]{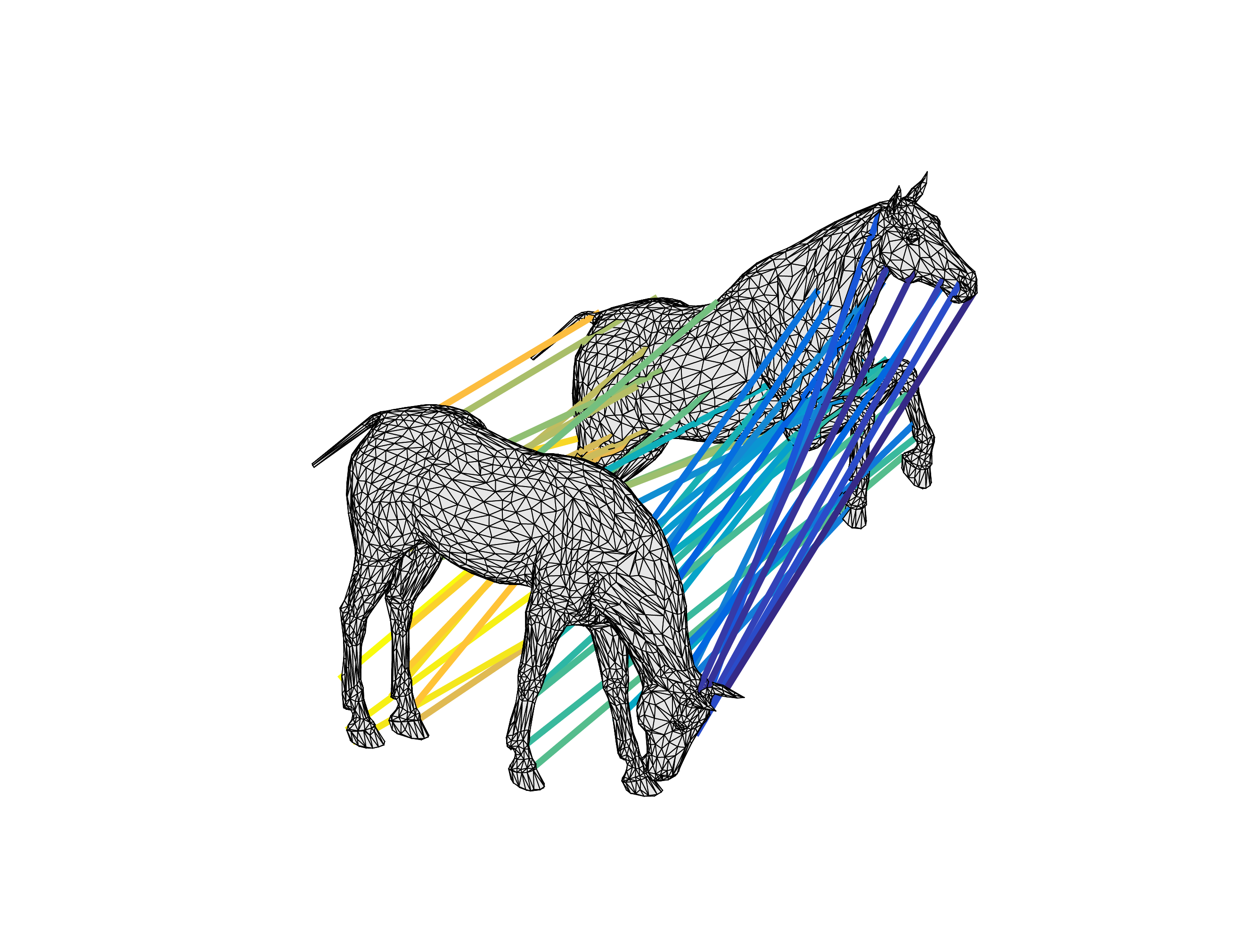}
\includegraphics[width=0.4\textwidth,  trim={5cm 8cm 5cm 8.5cm},clip]{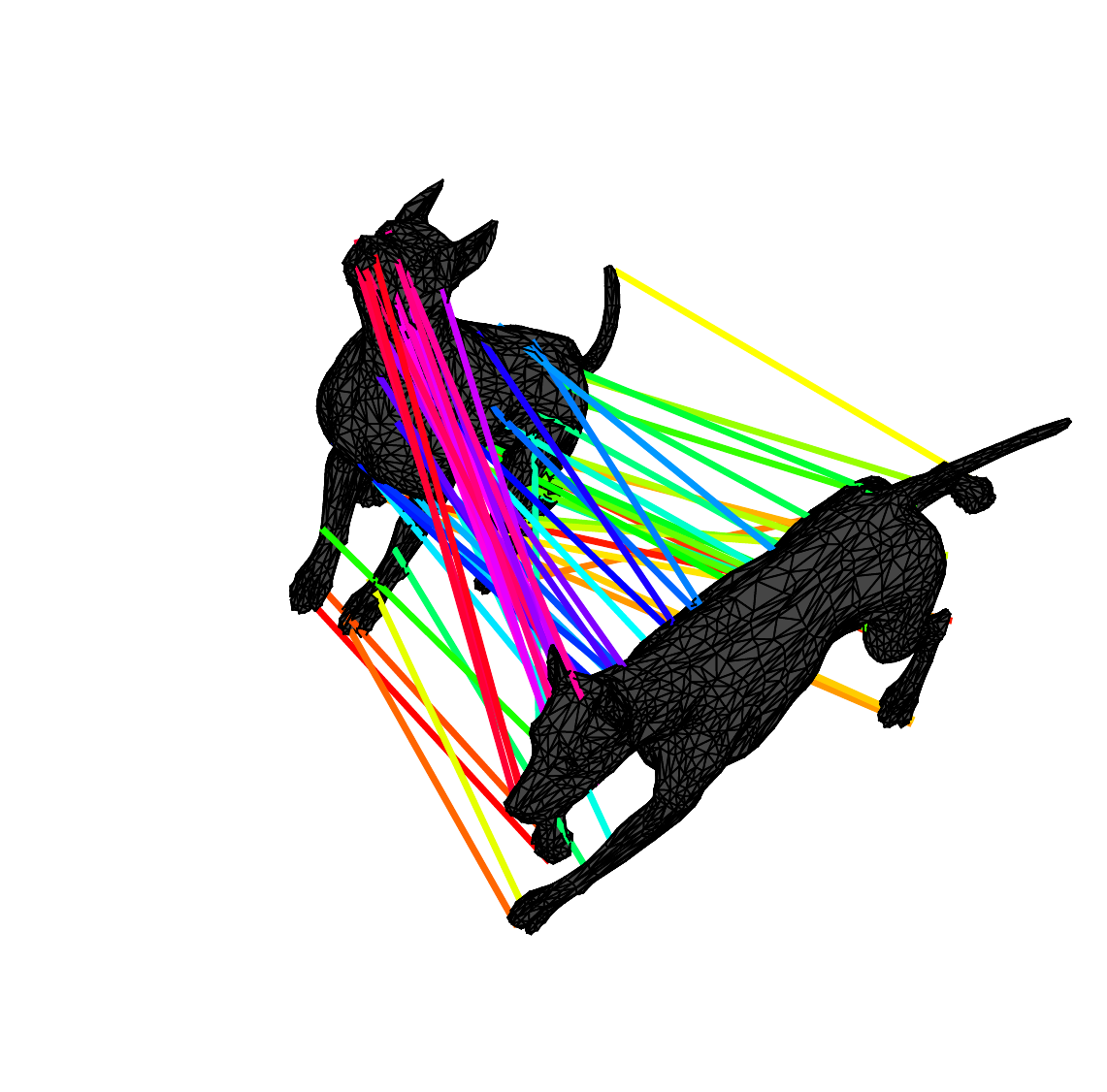}
\includegraphics[width=0.29\textwidth,  trim={5cm 6cm 5cm 8.2cm},clip]{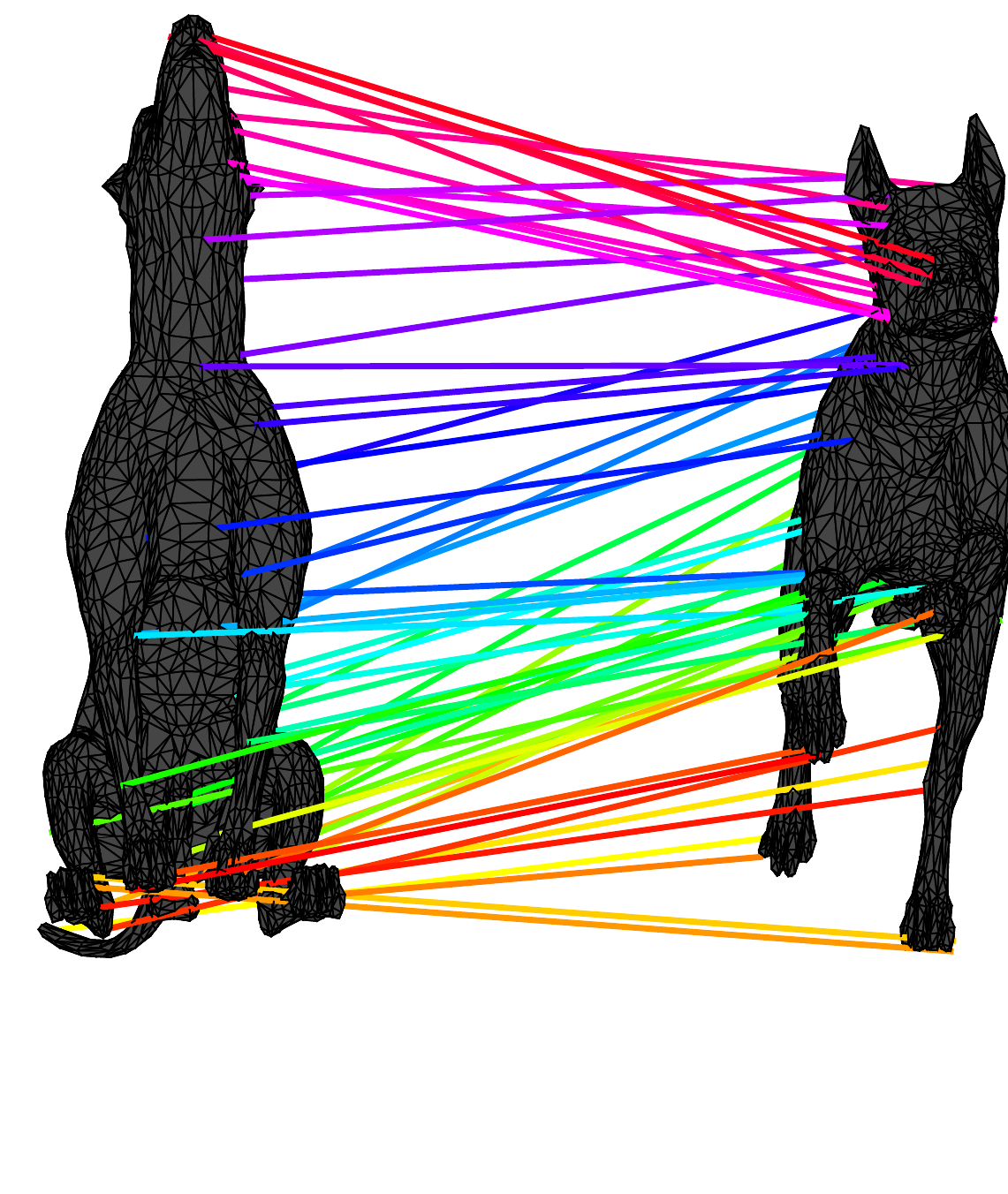}
\includegraphics[width=0.29\textwidth,  trim={5cm 6cm 5cm 8.5cm},clip]{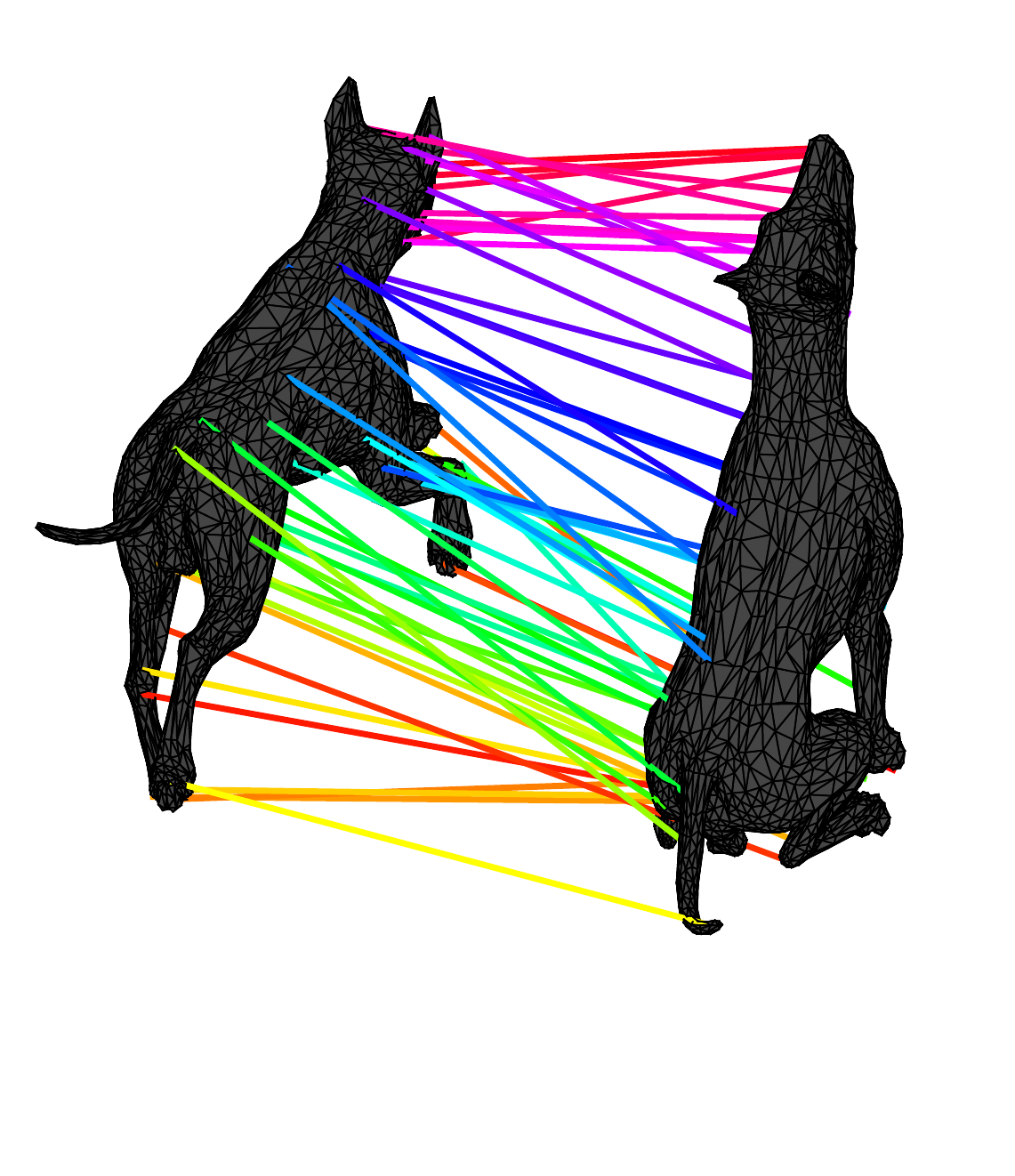}

\caption{\label{maps} We run GHMatch on 50 points sampled at random from the surfaces, and 60 points for the dogs. The pairwise distances we consider in each figure are the geodesic distances in the mesh. Images that are horizontally aligned correspond to different angles of the same correspondence between the 3D models. Note that for the dogs, the correspondence matches the left legs of one dog with the right legs of the other one (this is a consequence of the symmetry). Also note that there are small imperfections, like the tail of one dog matching with one leg of the other one (this is a consequence of randomly sampling and obtaining different number of points from different dogs tails). The algorithm with 50 sample points runs in less than 6 minutes on a standard 2013 MacBook Air.}
\end{figure}
In order to evaluate our non-convex optimization formulation of the
registration problem, we consider the 3D models from~\cite{Bronstein}
and we sample random points from each model.  We run the rank 1
augmented lagrangian optimization from Algorithm~\ref{algorithm},
using Matlab's implementation of the reflective trust region algorithm
to run the step 
\[
\mathbf y \gets \arg\min_{0\leq \mathbf y\leq 1}
\mathcal L(\mathbf y, \lambda_k, \sigma_k).
\]
In Figure~\ref{maps} we depict the resulting map between corresponding figures.

By design we know $\sigma_{k}\to \infty$ as $k$ increases, which
guarantees that $\|A\mathbf y_k -b\|\to 0$. However, there is no
theoretical guarantee that $\mathbf y_k$ will converge to a sparse
vector with entries in $\{0,1\}$.  However, we have observed in our numerical simulations
that $\mathbf y_k$ converges to a fairly sparse vector where most of
the entries are close to $0$ or $1$ provided a good choice of the
parameters $\mu$ and $\sigma_0$ (see Figure
\ref{error_plots}).  Moreover, regardless of the choice of the parameters, we
find that our thresholding step in the algorithm often obtains a map that
is bijective.  

\begin{remark}
As an alternative to the selection of parameters $\sigma_0$ and
$\lambda$, a thresholding step could be introduced inside the main
iteration (lines 5 to 9) to enforce sparsity of the resulting $\mathbf
y$.  
\end{remark}

\section*{Acknowledgements}
The authors would like to thank Facundo M\'emoli for insightful
discussions about the Gromov-Wasserstein distance, Tingran Gao for
pointing out useful references, Yaron Lipman for having his dataset
and code freely available, Stephen Wright for his advice on the
implementation of our non-convex algorithm and Marcos Cossarini for pointing out a misstatement in a previous version of this manuscript.  
\bibliography{gh}

\newcommand{\etalchar}[1]{$^{#1}$}
\begin{thebibliography}{OWWZ14}

\bibitem[Ban15]{Bandeira2015note}
Afonso~S Bandeira.
\newblock A note on probably certifiably correct algorithms.
\newblock {\em Comptes Rendus Mathematique, to appear}, 2015.

\bibitem[BBI01]{bbi}
Dmitri Burago, Yuri Burago, and Sergei Ivanov.
\newblock {\em A course in metric geometry}, volume~33.
\newblock American Mathematical Society Providence, 2001.

\bibitem[BBK08]{Bronstein}
Alexander Bronstein, Michael Bronstein, and Ron Kimmel.
\newblock {\em Numerical Geometry of Non-Rigid Shapes}.
\newblock Springer Publishing Company, Incorporated, 1 edition, 2008.

\bibitem[BLSC{\etalchar{+}}11]{Boyer}
Doug~M. Boyer, Yaron Lipman, Elizabeth St.~Clair, Jesus Puente, Biren~A. Patel,
  Thomas Funkhouser, Jukka Jernvall, and Ingrid Daubechies.
\newblock Algorithms to automatically quantify the geometric similarity of
  anatomical surfaces.
\newblock {\em Proceedings of the National Academy of Sciences},
  108(45):18221--18226, 2011.

\bibitem[CK14]{Clark02052014}
Connor Clark and Jugal Kalita.
\newblock A comparison of algorithms for the pairwise alignment of biological
  networks.
\newblock {\em Bioinformatics}, 2014.

\bibitem[DL16]{Lipman16}
Nadav Dym and Yaron Lipman.
\newblock Exact recovery with symmetries for procrustes matching.
\newblock {\em arXiv preprint arXiv:1606.01548}, 2016.

\bibitem[Ele12]{Elek}
G\'abor Elek.
\newblock Samplings and observables. limits of metric measure spaces.
\newblock arXiv:1205.6936, 2012.

\bibitem[Gro81]{gromov81}
Mikhail Gromov.
\newblock Groups of polynomial growth and expanding maps.
\newblock {\em Inst. Hautes \'Etudes Sci. Publ. Math.}, (53):53--73, 1981.

\bibitem[Gro01]{gromov}
Mikhail Gromov.
\newblock {\em Metric Structures for Riemannian and Non-Riemannian Spaces}.
\newblock Progress in Mathematics. Birkh{\"a}user, 2001.

\bibitem[IMPV15]{pcc}
Takayuki Iguchi, Dustin~G. Mixon, Jesse Peterson, and Soledad Villar.
\newblock Probably certifiably correct k-means clustering.
\newblock {\em arXiv:1509.07983}, 2015.

\bibitem[KKBL15]{kezurer15}
Itay Kezurer, Shahar~Z. Kovalsky, Ronen Basri, and Yaron Lipman.
\newblock Tight relaxation of quadratic matching.
\newblock {\em Computer Graphics Forum}, 34(5):115--128, 2015.

\bibitem[Lip11]{Lipman-website}
Yaron Lipman.
\newblock Software and teeth dataset, 2011.

\bibitem[LV09]{LottVillani}
John Lott and C{\'e}dric Villani.
\newblock Ricci curvature for metric-measure spaces via optimal transport.
\newblock {\em Ann. of Math. (2)}, 169(3):903--991, 2009.

\bibitem[MDK{\etalchar{+}}16]{Lipman}
Haggai Maron, Nadav Dym, Itay Kezurer, Shahar Kovalsky, and Yaron Lipman.
\newblock Point registration via efficient convex relaxation.
\newblock {\em ACM Trans. Graph.}, 35(4):73:1--73:12, July 2016.

\bibitem[M{\'{e}}m07]{memoli07}
Facundo M{\'{e}}moli.
\newblock {On the use of Gromov-Hausdorff Distances for Shape Comparison}.
\newblock pages 81--90, Prague, Czech Republic, 2007. Eurographics Association.

\bibitem[M{\'{e}}m11]{memoli11}
Facundo M{\'{e}}moli.
\newblock Gromov-{W}asserstein distances and the metric approach to object
  matching.
\newblock {\em Foundations of Computational Mathematics}, pages 1--71, 2011.
\newblock 10.1007/s10208-011-9093-5.

\bibitem[MS05]{memoli05}
Facundo M{\'{e}}moli and Guillermo Sapiro.
\newblock A theoretical and computational framework for isometry invariant
  recognition of point cloud data.
\newblock {\em Found. Comput. Math.}, 5(3):313--347, 2005.

\bibitem[NW06]{nocedal}
Jorge Nocedal and Stephen~J. Wright.
\newblock {\em Numerical Optimization}.
\newblock Springer, New York, 2nd edition, 2006.

\bibitem[OWWZ14]{odonell}
Ryan O'Donnell, John Wright, Chenggang Wu, and Yuan Zhou.
\newblock Hardness of robust graph isomorphism, lasserre gaps, and asymmetry of
  random graphs.
\newblock In {\em Proceedings of the Twenty-Fifth Annual ACM-SIAM Symposium on
  Discrete Algorithms}, pages 1659--1677. SIAM, 2014.

\bibitem[PCS16]{2016-peyre-icml}
Gabriel Peyr{\'e}, Marco Cuturi, and Justin Solomon.
\newblock Gromov-wasserstein averaging of kernel and distance matrices.
\newblock In {\em Proc. ICML'16}, pages 2664--2672, 2016.

\bibitem[Stu06]{Sturm}
Karl-Theodor Sturm.
\newblock On the geometry of metric measure spaces. {I}.
\newblock {\em Acta Math.}, 196(1):65--131, 2006.

\bibitem[Vil03]{villani}
C{\'{e}}dric Villani.
\newblock {\em Topics in optimal transportation}.
\newblock Graduate studies in mathematics. American Mathematical Society, cop.,
  Providence (R.I.), 2003.

\bibitem[YST15]{yang2015sdpnal+}
Liuqin Yang, Defeng Sun, and Kim-Chuan Toh.
\newblock Sdpnal+: a majorized semismooth newton-cg augmented lagrangian method
  for semidefinite programming with nonnegative constraints.
\newblock {\em Mathematical Programming Computation}, 7(3):331--366, 2015.

\end{thebibliography}
\bibliographystyle{alpha}
\end{document}